 \documentclass[11pt]{article}
\usepackage{amssymb, amsthm, amsmath, amscd}
\newtheorem{thm}{Theorem}[section]
\newtheorem{lem}[thm]{Lemma}
\newtheorem{prop}[thm]{Proposition}
\newtheorem{cor}[thm]{Corollary}

\theoremstyle{definition}\newtheorem{df}[thm]{Definition}
\theoremstyle{definition}\newtheorem{rem}[thm]{Remark}
\theoremstyle{definition}

\renewcommand{\phi}{\varphi}

\newcommand{\N}{\mathbb{N}}
\newcommand{\Z}{\mathbb{Z}}

\newcommand{\R}{\mathbb{R}}
\newcommand{\C}{\mathbb{C}}
\newcommand{\T}{\mathbb{T}}

\newcommand{\hm}{homomorphism}
\newcommand{\dt}{\delta}
\newcommand{\ep}{\epsilon}
\newcommand{\andeqn}{\,\,\,{\rm and}\,\,\,}
\newcommand{\rforal}{\,\,\,{\rm for\,\,\,all}\,\,\,}
\newcommand{\CA}{$C^*$-algebra}
\newcommand{\SCA}{$C^*$-subalgebra}

\newcommand{\af}{{\alpha}}
\newcommand{\bt}{{\beta}}

\newcommand{\beq}{\begin{eqnarray}}
\newcommand{\eneq}{\end{eqnarray}}
\newcommand{\tforal}{\,\,\,\text{for\,\,\,all}\,\,\,}
\newcommand{\tand}{\,\,\,\text{and}\,\,\,}

\title{Approximately Diagonalizing Matrices Over $C(Y)$
}
\author{Huaxin Lin
 }
\date{}

\begin{document}

\maketitle

\begin{abstract}
Let $X$ be a compact metric space which is locally absolutely
retract and let $\phi: C(X)\to C(Y, M_n)$ be a unital \hm, where $Y$
is a compact metric space with ${\rm dim}Y\le 2.$ It is proved that
there exists a sequence of $n$ continuous maps $\af_{i,m}: Y\to X$
($i=1,2,...,n$) and a sequence of sets of mutually orthogonal rank
one projections $\{p_{1, m}, p_{2,m},...,p_{n,m}\}\subset C(Y, M_n)$
such that
$$
\lim_{m\to\infty} \sum_{i=1}^n f(\af_{i,m})p_{i,m}=\phi(f) \rforal f\in C(X).
$$
 This is closely related to the Kadison diagonal matrix question.
It is also shown that this approximate diagonalization could not
hold in general when ${\rm dim}Y\ge 3.$
\end{abstract}

\section{Introduction}
Over two decades ago, Richard Kadison proved that a normal element
in $M_n({\cal N}),$ where ${\cal N}$ is a von-Neumann algebra, can
be diagonalized (\cite{RK1} and \cite{RK2})).  He showed that this
cannot be done if ${\cal N}$ is replaced by a unital \CA\, in
general. He then asked  what topological properties of a compact
metric space $Y$ will guarantee that
 every normal element in $M_n(C(Y))$ can always be diagonalized.
Karsten Grove and Gert K. Pedersen \cite{GP} showed that this could
not go very far. They demonstrated that $Y$ has to be sub-Stonean
and ${\rm dim}Y\le 2$ if every self-adjoint element can be
diagonalized in $M_n(C(Y)). $ Furthermore, they showed that, even
for sub-Stonean spaces $Y$ with ${\rm dim}Y\le 2,$ one still could
not diagonalize a normal element in general. In fact, they showed
that in order to have every normal element in $M_n(C(Y))$ to be
diagonalized, one must have that every finite covering space over
each closed subset of $Y$ is trivial and every complex line bundle
over each closed subset of $Y$ is trivial, in addition to  the
requirements that $X$ is sub-Stonean and ${\rm dim} Y\le 2.$ So not
every sub-Stonean space $X$ with dimension at most two has the
property that every normal element can be diagonalized. Since
sub-Stonean spaces are not the every-day topological space with
dimension at most two, it seems that the question of diagonalizing
normal elements in $M_n(C(Y))$ has a rather negative answer.

However, in the decades after the original question was raised and
answered, it seems that approximately diagonalizing some normal
elements or some commutative \SCA s in $M_n(C(Y)),$ where $Y$ is a
lower dimensional nice topological space, becomes quite useful and
important. In this paper, instead of considering exact
diaogonalization of commutative \SCA s in $M_n(C(Y)),$ we study the
problem whether a unital \hm\, $\phi: C(X)\to M_n(C(Y))$ can be
approximately diagonalized. To be precise, we formulate as follows:
Let $\ep>0$ and a compact set ${\cal F}\subset C(X)$ be given. Are
there continuous maps $\af_i: Y\to X$ ($1\le i\le n$) and mutually
orthogonal rank one projections $p_1, p_2,...,p_n\in M_n(C(Y))$ such
that
\beq\label{I1}
\|\phi(f)-\sum_{i=1}^n f(\af_i)p_i\|<\ep\tforal f\in {\cal F}?
\eneq
Here rank one projections are those projections in $M_n(C(Y))$ for
which every fiber has rank one. Note that we do not insist that
$p_1, p_2,...,p_n$ are equivalent. Note also that if $a\in
M_n(C(Y))$ is a normal element and if $sp(a)=X,$ a compact subset of
the plane, then $a$ induces a unital  \hm\, $\phi: C(X)\to
M_n(C(Y))$ by defining $\phi(f)=f(a)$ for all $f\in C(X).$ However,
we study the general case that  $X$ is a compact metric space.

It has been known that those unital \hm s which can be approximately
diagnalized are very useful, for example, in the study of inductive
limits of homogeneous \CA s, a subject has profound impact in the
program of classification of amenable \CA s, otherwise known as the
Elliott program.

The main result that we report here is that the answer to (\ref{I1})
is affirmative for any compact metric space $X$ which is locally
absolutely retract (see \ref{LAR}) and any compact metric space $Y$
with ${\rm dim} Y\le 2.$ Moreover, we show that, the answer is
negative for general compact metric space $Y$ with ${\rm dim}\ge 3.$
In fact, a unitary in $M_2(C(S^3))$ may not be approximately
diagonalized. We also show that if ${\rm dim} Y>3,$ then there are
self-adjoint elements with spectrum $[0,1]$ in $M_n(C(Y))$ which can
not be  approximately diagonalized. For more general compact metric
space $X,$ we show that, for any $\ep>0,$ any compact subset ${\cal
F}\subset C(X),$ there is a unital commutative diagonal \SCA\,
$B\subset M_n(C(Y))$ such that
$$
{\rm dist}(\phi(f), B)<\ep\tforal f\in {\cal F},
$$
provided that ${\rm dim Y}\le 2$
(see \ref{M1}).

As expected, when one studies unital \hm s from $C(X)\to M_n(C(Y)),$
one often  needs to make some perturbation. The trouble  arises when
one tries to approximate an `almost \hm' by a \hm. This is
already problematic when $Y$ is just a point. Suppose that $X$ is a
compact subset of the plane and $L: C(X)\to M_n$ is unital positive
linear map which is almost multiplicative. It was first proved by
W.A.J. Luxembourg and F.R. Taylor \cite{LT}, using non-standard
analysis, that such maps can be approximated by \hm s. Theorem
\ref{d0semi} generalizes this to the case that $X$ is any compact
metric space. Theorem \ref{d1semi} shows that the same statement
holds when $M_n$ is replaced by $C([0,1], M_n).$

The paper is organized as follows. In section two, we present a
theorem which generalizes the early result of Luxembourg and Taylor
mentioned above. In section three, we collect some easy facts. Lemma
\ref{d0uniq} serves as a uniqueness theorem in finite dimensional
\CA s and Lemma \ref{frombk} may be viewed as an elementary version
of the so-called Basic Homotopy Lemma (see \cite{BEEK} and
\cite{Lnmemoir}). In section four, we provide a uniqueness theorem
for \hm s from $C(X)$ into $C([0,1], M_n)$ which extends Lemma
\ref{d0uniq}. In section five, using the results in the previous
sections, we first show that an ``almost \hm" from $C(X)$ into
$C([0,1], M_n)$ is close to a true \hm, further generalizing the
theorem of Luxembourg and Taylor. Then we present a version of Basic
Homotopy Lemma in $C([0,1], M_n).$  Section six contains the main
result which based on the previous sections, in particular, the
uniqueness theorem in section four and the Basic Homotopy Lemma in
section five. Finally, in section seven, we show that one should not
expect same results when ${\rm dim Y}\ge 3.$

{\bf Acknowledgements}: Most of this work was done when the author
was in East China Normal University in the summer 2009. This work was
partially supported by a NSF grant, Changjiang Professorship from
ECNU and Shanghai Priority Academic Disciplines.

\section{Approximate \hm s}

\begin{lem}\label{d0finite}
Let $X$ be a compact metric space and let $n\ge 1$ be an integer.
Let $H: C(X)\to C^b({\mathbb N},M_n)/C_0({\mathbb N}, M_n)$ be a unital \hm.
Then there exists an infinite  subsequence $S\subset  \N$ such that
the induced \hm\, $H': C(X)\to C^b(S,M_n)/C_0(S,M_n)$ has finite spectrum.

\end{lem}

\begin{proof}
Put $A_0=C^b({\mathbb N},M_n)/C_0({\mathbb N}, M_n).$
Denote by $\pi: C^b({\mathbb N},M_n)\to A_0$  the quotient map.
Let $\xi_1\in X$ be a point in the spectrum of $H.$
Let
$$
I_1=\{f\in C(X): f(\xi_1)=0\}.
$$
Then $\overline{H(I_1)A_0H(I_1)}$ is a $\sigma$-unital hereditary
\SCA\, and it is not $A_0,$ since $\xi_1$ is in the spectrum of $H$
and $H(I_1)$ is a proper closed ideal of $H(C(X)).$ Note that $A_0$
is the corona algebra of the separable \CA\, $C_0({\mathbb N},
M_n).$ It follows from a theorem of G.Pedersen (Th.15 of \cite{P})
that $\overline{H(I_1)A_0H(I_1)}^{\perp}\not=\{0\}.$ Since
$\overline{H(I_1)A_0H(I_1)}^{\perp}$ is a hereditary \SCA\, of $A_0$
and $A_0$ has real rank zero, there is a non-zero projection $p_1\in
\overline{H(I_1)A_0H(I_1)}^{\perp}.$ It follows that
\beq\label{d0finite-1}
H(f)p_1=f(\xi_1)p_1\tforal f\in C(X).
\eneq
It is standard that there exists a sequence of projections $\{p_1(m)\}\subset M_n$ such that
$\pi(\{p_{1}(m)\})=p_{1}.$ Let $S_1\subset \mathbb N$ be the subsequence so that $p_{1}(m)\not=0$ for all $m\in S_1.$
Note that $S_1$ must be infinite.
Let $A_1=A_0/J_1,$ where
$$
J_1=\{\{a_m\}\in C^b({\mathbb N}, M_n): a_m=0\rforal m\in S_1\}.
$$
One also has that $A_1\cong C^b(S_1, M_n)/C_0(S_1, M_n).$
Let $\Phi_1: A_0\to A_1$ be the quotient map and define $H_1=\Phi_1\circ H.$
If $\xi_1$ is the only point in the spectrum of $H_1,$ the lemma follows. Otherwise, let $\xi_2\not=\xi_1$ be another point
in the spectrum of $H_1.$
Let
$$
I_2=\{f\in C(X): f(\xi_2)=0\}.
$$
From the above argument, one obtains a nonzero projection $p_2\in
\overline{H_1(I_2)A_1H_1(I_2)}^{\perp}.$ Then $\Phi_1(p_1)p_2=0$ and
\beq\label{d0finite-2}
H_1(f)p_2=f(\xi_2)p_2\rforal f\in C(X).
\eneq
There exists a projection $\{p_2(m)\}\in C^b(S_1, M_n)$
such that $\pi_1(\{p_2(m)\})=p_2$ and
\beq\label{d0finite-3}
p_2(m)p_1(m)=0
\eneq
for all $m,$  where $\pi_1: C^b(S_1,M_n)\to A_1$ is the
quotient map. Let $S_2\subset S_1$ be such that $p_2(m)\not=0$ for
all $m\in S_2.$ Then $S_2$ is an infinite subset. Let
$$
J_2=\{\{a_m\}\in C^b(S_1,M_n): a_m=0\rforal m\in S_2\}.
$$
Put $A_2=A_1/J_2$ and let $\Phi_2: A_1\to A_2$ be the quotient map.
Note that $A_2\cong C^b(S_2,M_n)/C_0(S_2,M_n).$ Moreover,
\beq\label{d0finite-4}
p_1(m)\not=0\andeqn p_2(m)\not=0 \tforal m\in S_2.
\eneq
Define $H_2=\Phi_2\circ H_1.$ Then $\xi_1, \xi_2$ are in the
spectrum of $H_2.$ If the spectrum of $H_2$ contains only these two
points, the lemma follows. Otherwise, we continue. However, since
there can be no more than $n$ mutually orthogonal non-zero
projections in $M_n,$ from (\ref{d0finite-3}) and
(\ref{d0finite-4}), this process has to stop at the stage $n$ or
earlier. At that point, one obtains an infinite subset $S\subset
\mathbb N,$ for which $H': C(X)\to C^b(S,M_n)/C_0(S,M_n)$ has finite
spectrum.

\end{proof}

\begin{thm}\label{d0semi}
Let $X$ be a compact metric space, let $n\ge 1$ be an integer and
let $M>0.$  For any $\ep>0$ and any finite subset ${\cal F}\subset
C(X),$ there exists $\dt>0$ and a finite subset ${\cal G}\subset
C(X)$ satisfying the following: for any unital map  $\phi: C(X)\to
M_n$  with $\|\phi(f)\|\le M$ for all $ f\in C(X)$ with $\|f\|\le 1$
such that
\beq\label{d0semi-1}
\|\phi(\lambda_1 x+ \lambda_2 y)-(\lambda_1\phi(x)+\lambda_2\phi(y))\|<\dt,\\
\|\phi(xy)-\phi(x)\phi(y)\|<\dt\tand \|\phi(x^*)-\phi(x)^*\|<\dt
\eneq
for all $\lambda_1,\lambda_2\in {\mathbb C}$ with $|\lambda_i|\le 1$
($i=1,2$)
 and $x, y\in {\cal G},$ there exists a unital
\hm\, $\psi: C(X)\to M_n$ such that
\beq\label{d0semi-2}
\|\phi(f)-\psi(f)\|<\ep\tforal f\in {\cal G}.
\eneq

\end{thm}

\begin{proof}
Let $H(C(X), M_n)$ be the set of unital \hm s. Suppose that the theorem fails.
There exists $\ep_0>0$ and a finite subset ${\cal G}_0\subset C(X)$ with the following properties:

There exists a sequence of unital maps $\phi_k: C(X)\to M_n$ with
$\|\phi_k(f)\|\le M$ for all $\|f\|\le 1$ ($k=1,2,...$) such that
\beq\label{d0semi-3}
\lim_{k\to\infty}\|\phi_k(\lambda_1x+\lambda_2y)-(\lambda_1\phi_k(x)+\lambda_2\phi_k(y))\|=0,\\\label{d0semi-4}
\lim_{k\to\infty}\|\phi_k(xy)-\phi_k(x)\phi_k(y)\|=0\tand
\lim_{k\to\infty}\|\phi_k(x^*)-\phi_k(x)^*\|=0
\eneq
for all $x, y\in C(X)$ and $\lambda_1, \lambda_2\in {\mathbb C}$ with $|\lambda_i|\le 1,$ $i=1,2,$ but
\beq\label{d0semi-5}
\inf_{\psi\in H(C(X), M_n)}\{\inf_k\sup\{\|\phi_k(f)-\psi(f)\|: f\in
{\cal G}_0\}\}\ge \ep_0,
\eneq
where $H(C(X), M_n)$ is the set of unital \hm s from $C(X)$ to
$M_n.$  Let $A=C^b({\mathbb N},M_n),$ $I=C_0({\mathbb N},M_n)$ and
let $\pi: A\to A/I$ be the quotient map. Define $L: C(X)\to A$ by
$L(f)=\{\phi_k(f)\}_{k\in {\mathbb N}}$ for $f\in C(X)$ and
$H=\pi\circ L.$ From (\ref{d0semi-3}) and (\ref{d0semi-4}), $H:
C(X)\to A/I$ is a unital \hm.  It follows from \ref{d0finite} that
there exists an infinite subset $S\subset \mathbb N$ such that $H_1:
C(X)\to C^b(S,M_n)/C_0(S,M_n)$ defined by $H_1=\Phi\circ H$ has
finite spectrum, where $\Phi: A\to A_1=A/I_1$ is the quotient map
and where
$$
I_1=\{\{a_m\}\in C^b({\mathbb N},M_n): a_m=0\rforal m\in S\}.
$$
By passing to a subsequence, without loss of generality,  one may assume
that $H$ has finite spectrum. Therefore there are mutually orthogonal projections $\{p_1, p_2,...,p_K\}\subset A$ such that
\beq\label{d0semi-6}
H(f)=\sum_{j=1}^K f(\xi_j)p_j\rforal f\in C(X),
\eneq
where $\{\xi_1, \xi_2,...,\xi_k\}\subset X$ is a finite subset.
There are mutually orthogonal projections $P_1, P_2,...,P_K\in
C^b({\mathbb N},M_n)$ such that $\pi(P_j)=p_j,$ $j=1,2,...,K.$
Let $P_j=\{q_j(m)\},$ where each $q_j(m)$ is a projection and
$q_j(m)q_i(m)=0,$ if $i\not=j,$ $i,j,=1,2,...,K.$ Define $\psi_m:
C(X)\to M_n$ by $\psi_m(f)=\sum_{j=1}^K f(\xi_j)q_j(m)$ for all
$f\in C(X).$ Then $\pi\circ \{\psi_m\}=H.$ It follows that
\beq\label{d0semi-7}
\lim_{m\to\infty}\|\phi_m(f)-\psi_m(f)\|=0\rforal f\in C(X).
\eneq
Hence, there exists an integer $N\ge 1$ such that
$\|\phi_m(f)-\psi_m(f)\|<\ep_0/2\tforal f\in {\cal F}.$
This contradicts with (\ref{d0semi-5}). The lemma follows.

\end{proof}

\begin{cor}\label{ntuple}
Let $k, n\ge 1$ be two integers and let $\ep>0.$ Then there exists
$\dt>0$ satisfying the following: Suppose that $x_1, x_2,...,x_k\in
M_n$ are $k$ self-adjoint elements with $\|x_i\|\le 1$
($i=1,2,...,k$) for which
$$
\|x_ix_j-x_jx_i\|<\dt,\,\,\,i=1,2,...,k.
$$
Then there are $k$ self-adjoint elements $y_1,y_2,...,y_k\in M_n$
with $\|y_i\|\le 1$ ($i=1,2,...,k$) such that
\beq\label{ntuple-1}
y_iy_j=y_jy_i\tand \|x_i-y_i\|<\ep,\,\,\,i=1,2,...,k.
\eneq
\end{cor}

\begin{proof}
This follows from \ref{d0finite} as in the proof of \ref{d0semi}.
One  sketches here. If the corollary fails, there would be an
 $\ep_0>0$ and a sequence of $k$ self-adjoint
elements $\{x_j^{(m)}\},$ $j=1,2,...,k,$ such that
\beq\label{cnt-1}
\lim_{m\to \infty}\|x_j^{(m)}x_i^{(m)}-x_i^{(m)}x_j^{(m)}\|=0,\,i,j=1,2,...,k, \andeqn\\
\inf\{\inf_m\{\max\{\|x_j^{(m)}-y_j\|: 1\le j\le m\}\}\}\ge \ep_0,
\eneq
where the outside infimum is taken among all possible commuting
$k$-tuple of self-adjoint elements $\{y_j: 1\le j\le k\}$ in $M_n.$
Let $z_j=\pi(\{y_j^{(m)}\})\in C^b({\mathbb N}, M_n),$ where  $\pi:
C^b({\mathbb N}, M_n) \to C^b({\mathbb N}, M_n)/C_0({\mathbb N},
M_n)$ is the quotient map. Then $z_jz_i=z_iz_j,$ $i,j=1,2,...,k.$
Let $\Omega=\{(r_1,r_2,...,r_k)\in \C^k: |r_j|\le 1, 1\le j\le k\}.$
Define $\phi: C(\Omega)\to C^b({\mathbb N}, M_n)/C_0({\mathbb N},
M_n)$ by $\phi(f)=f(z_1,z_2,...,z_k)$ for all $f\in C(\Omega).$ It
follows from \cite{CE} that there exists a unital completely
positive linear map $L: C(\Omega)\to C^b({\mathbb N}, M_n)$ such
that $\pi\circ L=\phi.$ One may write $L(f)=\{L_m(f)\}_{m\in
{\mathbb N}}.$ One then applies \ref{d0semi} to $L_m$ (for all
sufficiently large $m$) to obtain unital \hm s $\phi_m: C(\Omega)\to
M_n$ such that $\lim_{m\to\infty}\|\phi_m(f)-L_m(f)\|=0$ for all
$f\in C(\Omega).$ A contradiction would be reached as in the proof
of \ref{d0semi}.


\end{proof}

\begin{rem}
 Note that, in \ref{d0semi}, $\dt$
depends on $X,$ $\ep$ as well as $n.$  In Corollary \ref{ntuple},
$\dt$ depends on both $k$ and $n.$ In the case that $k=2,$ the
theorem for which $\dt$ does not depend on $n$ was first proved in
\cite{LnFields}. That result is much deeper and was false if $k\ge
3$ (see \cite{GL}).
\end{rem}

\section{Commutative \SCA s of matrix algebras}


\begin{lem}\label{d0uniq}
Let $X$ be a compact metric space and let $n\ge 1$ be an integer.
Then, for any $\ep>0$ and any finite subset ${\cal F}\subset C(X),$ there exist a finite subset ${\cal G}$ and
$\dt>0$ satisfying the following. Let $\phi, \psi: C(X)\to M_n$ be two unital \hm s for which
\beq\label{d0uniq-1}
|{\rm tr}\circ \phi(g)-{\rm tr}\circ \psi(g)|<\dt\tforal g\in {\cal G},
\eneq
where ${\rm tr}$ is the normalized tracial state. Then there exists a unitary $u\in M_n$ such that
\beq\label{d0uniq-2}
\|{\rm ad}\, u\circ \phi(f)-\psi(f)\|<\ep\tforal f\in {\cal F}.
\eneq
\end{lem}

\begin{proof}
Let $\eta>0$ be such that, for any $x, x'\in X,$
$$
|f(x)-f(x')|<\ep\rforal f\in {\cal F},
$$
provided ${\rm dist}(x,x')<\eta.$ Let $\{x_1, x_2,...,x_m\}$ be an
$\eta/8$-dense subset of $X.$ For each subset $F\subset
\{x_1,x_2,...,x_m\},$ define  $g_F\in C(X)$ to be a function with
$0\le g_F(x)\le 1$ for all $x\in X,$ $g_F(x)=1$ if $x\in O(F,
\eta/4),$ and $g_F(x)=0,$ if $x\not\in O(F,\eta/2).$

Let $\dt=1/2n$ and let ${\cal G}=\{g_F: F\subset
\{x_1,x_2,...,x_m\}\}.$  Suppose that $\phi$ and $\psi$ satisfy the
assumption for the above $\dt$ and ${\cal G}.$ Let $\{\xi_1,
\xi_2,...,\xi_n\}$ and $\{\zeta_1, \zeta_2,...,\zeta_n\}\subset X$
be such that
$$
\phi(f)=\sum_{i=1}^n f(\xi_i)p_i\andeqn \psi(f)=\sum_{i=1}^n
f(\zeta_i)q_i
$$
for all $f\in C(X),$ where $\{p_1, p_2, ...,p_n\}$ and $\{q_1,
q_2,...,q_n\}$ are two sets of mutually orthogonal rank one
projections. Fix  a  subset $S\subset \{\xi_1, \xi_2,...,\xi_n\}$ of
$k$ ($1\le k\le n$) elements. For each $\xi_j\in S,$ $\xi_j\in
B(x_i,\eta/8)$ for some $i.$  Let $F=\{ x_i: {\rm dist}(x_i,
S)<\eta/8\}.$ Then
\beq\label{d0uniq-3}
k/n &\le &\mu_{{\rm tr}\circ \phi}(O(S, \eta/8))\le \mu_{\rm
tr}\circ \phi(O(F,\eta/4))\\\label{d0uniq-3+}
 &\le & {\rm tr}\circ
\phi(g_F) < {\rm tr}\circ \psi(g_F)+\dt \le \mu_{{\rm tr}\circ
\psi}(O(F, \eta/2)) +1/2n.
\eneq
Note that
$$
\mu_{{\rm tr}\circ \psi}(O(F, \eta/2))=k_1/n
$$
for some non-negative integers $k_1.$  Thus, from (\ref{d0uniq-3})
and (\ref{d0uniq-3+}),
$$
k/n<k_1/n+1/2n.
$$
It follows that $k\le k_1.$ Thus
\beq\label{d0uniq-4}
\mu_{{\rm tr}\circ \phi}(O(S, \eta/8))\le  \mu_{{\rm tr}\circ
\psi}(O(F, \eta/2))
\eneq
for any $S\subset \{\xi_1,\xi_2,...,\xi_n\}$ of $k$ elements. This
implies that there is a finite subset $S_1\subset \{\zeta_1,
\zeta_2,...,\zeta_n\}$ of at least $k$ elements such that
$S_1\subset O(F,\eta/2).$
It follows that, for each $\zeta_j\in S_1,$ there is $x_i\in F$ such
that
$
\xi_i\in O(x_i, \eta/2).
$
By the definition of $F,$ there exists $\xi_l\in S$ such that
$\zeta_j\in O(\xi_l, 5\eta/8).$

In other words, $O(S, 5\eta/8)$ contains at least $k$ elements of
$\{\zeta_1, \zeta_2,...,\zeta_n\}.$
 Therefore, by the Marriage Law (see \cite{HV}),
 there exists  a permutation $\gamma: (1,2,...,n)\to (1,2,...,n)$ such that
 $$
 {\rm dist}(\xi_i, \zeta_{\gamma(i)})<\eta, \,\,\,i=1,2,...,n.
 $$
Let $U$ be the unitary such that
$
U^*p_iU=q_{\gamma(i)},\,\,\,i=1,2,...,n.
$
Then
$$
U^*\phi(f)U=\sum_{i=1}^n f(\xi_i)q_{\gamma(i)}\rforal f\in C(X).
$$
It follows that
$$
\|U^*\phi(f)U-\psi(f)\|<\ep\rforal f\in {\cal F}.
$$

\end{proof}


\begin{lem}\label{d0uniqhm}
Let $X$ be a compact metric space which is locally path connected
and let $n\ge 1.$ Then, for any $\ep>0,$ $\ep_1>0$ and any finite
subset ${\cal F}\subset C(X),$ there exist a finite subset ${\cal
G}$ and $\dt>0$ satisfying the following. Let $\phi, \psi: C(X)\to
M_n$ be two unital \hm s for which
\beq\label{d0uniqhm-1}
|{\rm tr}\circ \phi(g)-{\rm tr}\circ \psi(g)|<\dt\tforal g\in {\cal G},
\eneq
where ${\rm tr}$ is the normalized tracial state. Then there exists a unital \hm\, $\Phi: C(X)\to C([0,1], M_n)$ such that
\beq\label{d0unihm-2}
\Phi(f)(0)=\phi(f)\tand \|\Phi(f)(t)-\phi(f)\|<\ep\tforal f\in {\cal
F}\andeqn t\in [0,1],
\eneq
and there exists
a unitary $u\in M_n$ such that
\beq\label{d0uniqhm-3}
{\rm ad}\, u\circ \Phi(f)(1)=\psi(f) \tforal f\in C(X).
\eneq
Moreover, there are continuous maps $\af_i: [0,1]\to X$
$(i=1,2,...,n$) and mutually orthogonal rank one projections
$\{p_1,p_2,...,p_n\}\subset M_n$ such that
\beq\label{d0uniqhm-3+}
\Phi(f)=\sum_{i=1}^n f(\af_i)p_i\rforal f\in C(X)\andeqn \\
{\rm dist}(\af_i(t), \af_i(0))<\eta\tforal t\in [0,1],\,\,\,
i=1,2,...,n.
\eneq
\end{lem}

\begin{proof}
This follows from the proof of \ref{d0uniq}.  At the beginning of the proof of
\ref{d0uniq},  we may further require that there is $\eta_1>0$ such that  each open ball $B_{\eta_1}$ with radius $\eta_1$ is contained in a path connected neighborhood $Z_{\eta}\subset B_{\eta}.$
Let $\dt$ be as in the proof of \ref{d0uniq} and choose ${\cal G}$ as in the proof of \ref{d0uniq}
(but for $\eta_1$ instead of $\eta$).  Write
$\phi(f)=\sum_{i=1}^nf(\xi_i)p_i$ and $\psi(f)=\sum_{i=1}^n f(\zeta_i)q_i$ for all $f\in C(X)$ as in the proof of \ref{d0uniq}.
The proof of \ref{d0uniq} provides a permutation $\gamma: (1,2,...,n)\to (1,2,...,n)$ such that
\beq\label{d0uniqhm-4}
{\rm dist}(\xi_i, \zeta_{\gamma(i)})<\eta_1.
\eneq
Since each open ball $B(\xi_i, \eta_1)\subset Z_{i, \eta_2}$ for
some path connected neighborhood contained in $B(x_i, \eta_2),$
where $\eta_2=\min\{\ep_1, \eta\},$  there exist a continuous path
$\alpha_i: [0,1]\to Z_{i, \eta_2}\subset B(\xi, \eta_2)$ such that
$$
\alpha_i(0)=\xi_i\andeqn \alpha_i(1)=\zeta_i,\,\,\,i=1,2,...,n.
$$
Define $\Phi: C(X)\to C([0,1], M_n)$ by
$
\Phi(f)(t)=\sum_{i=1}^n f(\alpha_i(t))p_i\tforal f\in C(X)
$
and $t\in [0,1].$ Since $\alpha_i(t)\in Z_{i, \eta_2}\subset
B(\xi_i, \eta_2),$ one estimates that
\beq\label{d0uniqhm-5}
\|\Phi(f)(t)-\phi(f)\|<\ep\tforal f\in {\cal F}
\eneq
and $t\in [0,1].$ Moreover,
\beq\label{d0uniqhm-6}
\Phi(f)(1)=\sum_{i=1}^nf(\zeta_i)p_i\tforal f\in C(X).
\eneq
Since $\{q_i: i=1,2,...,n\}$ and $\{p_i: i=1,2,...,n\}$ are assumed to be two sets of mutually orthogonal
rank one projections, as in the proof of \ref{d0uniq}, one obtains a unitary $u\in M_n$ such that
\beq\label{d0uniqhm-7}
u^*\Phi(f)(1)u=\phi(f)\tforal f\in C(X).
\eneq

\end{proof}

\begin{lem}\label{frombk}
Let $\ep>0,$ $n\ge 1$ be an integer and $M>0.$ There exists $\dt>0$
satisfying the following: For any finite subset ${\cal F}\subset
M_n$ with $\|a\|\le M$ for all $a\in {\cal F}$ and a unitary $u\in
M_n$ such that
$$
\|ua-au\|<\dt\tforal a\in {\cal F},
$$
there exists a continuous path of unitaries $\{u(t):t\in
[0,1]\}\subset M_n$ with $u(0)=u$ and $u(1)=1$ such that
$$
\|u(t)a-au(t)\|<\ep\tforal a\in {\cal F}.
$$
Moreover,
$$
{\rm Length}(\{u(t)\})\le 2\pi.
$$
\end{lem}

\begin{proof}
Note that $\T\setminus sp(u)$ contains an arc with length at least $2\pi/n.$ Thus the lemma follows immediately from
Lemma 2.6.11 of \cite{Lnbk}.

\end{proof}

\section{Commutative \SCA s in matrix algebras over one dimension spaces}
\begin{lem}\label{d1uniq}
Let $X$ be a path connected
finite CW
complex, let $C=C(X)$ and let $A=C([0,1],M_n).$ For any $\ep>0$ and
any finite subset ${\cal F}\subset C,$ there exists a finite subset
${\cal G}\subset C$ and $\dt>0$ satisfying the following: Let $\phi,
\psi: C\to A$ be two unital \hm s such that
\beq\label{d1uni-1}
|\tau\circ \phi(g)-\tau\circ \psi(g)|<\dt\tforal g\in {\cal G} \tand \tforal \tau\in T(A).
\eneq

Then there exists a unitary $U\in A$ such that
\beq\label{d1uni-2}
\|{\rm ad}\,U\circ \phi(f)-\psi(f)\|<\ep
\eneq
for all $f\in {\cal F}.$

Moreover, if, in addition, $\phi(f)(0)=\psi(f)(0),$ or
$\phi(f)(0)=\psi(f)(0)$ and $\phi(f)(1)=\psi(f)(1)$ for all $f\in
C(X),$ then one may assume that $U(0)=1_{M_n},$ or $U(0)=U(1)=1,$ respectively.
\end{lem}

\begin{proof}
Without loss of generality, we may assume that ${\cal F}$ is in the
unit ball of $C(X).$ Put $d=2\pi/n.$ Let $\dt_0>0$ (in place of
$\dt$) be as required by Lemma 2.6.11 of \cite{Lnbk} for $\ep/4.$
Let $0<\dt_1<1/2n$ (in place of $\dt$) and ${\cal G}_1$ (in place of
${\cal G}$) associated with ${\cal F}$ and $\min\{\ep/8, \dt_0/4\}$
(in place of $\ep$) as required by \ref{d0uniqhm}.


Put $\ep_1=\min\{\ep/16, \dt_1/4, \dt_0/4\}.$ There exists $\eta>0$ such that
$$
|g(t)-g(t')|<\ep_1/2\tforal g\in \phi({\cal G}_0\cup {\cal F})\cup
\psi({\cal G}_0\cup {\cal F}).
$$
provided that $|t-t'|<\eta_1.$ Choose a partition of the interval:
$$
0=t_0<t_1<\cdots <t_N=1
$$
such that $|t_i-t_{i-1}|<\eta_1,$ $i=1,2,...,N.$ Then
\beq\label{d1uni-3}
\|\phi(f)(t_i)-\phi(f)(t_{i-1})\|<\ep_1\andeqn \|\psi(f)(t_i)-\psi(f)(t_{i-1})\|<\ep_1
\eneq
for all $f\in {\cal G}_0\cup {\cal F},$ $i=1,2,...,N.$
There are unitaries $U_i\in M_n$ and $\{x_{i,j}\}_{j=1}^n$ such that
$$
\phi(f)(t_i)=U_i^*\begin{pmatrix} f(x_{i,1}) & &\\
                       & \ddots &\\
                        && f(x_{i,n})\end{pmatrix} U_i,
                        $$
 $i=0,1,2,...,N.$     It follows from \ref{d0uniq} that there exists, for each $i,$  a unitary $W_i \in M_n$ such that
 \beq\label{d1uni-4}
\| W_i^* \phi(f)(t_i)W_i-  \psi(f)(t_i)\|<\min\{\ep/8, \dt_0/4\}
\eneq
for all $f\in {\cal G}_0\cup {\cal F}.$ We estimate, from
(\ref{d1uni-3}), that
\beq\label{d1uni-5}
\|\phi(f)(t_i)-W_{i-1}W_i^*\phi(f)(t_i)W_iW_{i-1}^*\|<\ep_1+2\min\{\ep/8, \dt_0/4\}\tforal f\in {\cal G}_0\cup {\cal F},
\eneq
$i=1,2,...,N.$ By the choice of $\ep_1$ and applying Lemma 2.6.11 of \cite{Lnbk},
we obtain $h_i\in (M_n)_{s.a.}$ such that $W_iW_{i-1}^*=\exp(\sqrt{-1} h_i),$
\beq\label{d1uni-6}
&&\|h_i\phi(f)(t_i)-\phi(f)(t_i) h_i\|<\ep/4\andeqn\\
&&\|\exp(\sqrt{-1} t h_i)\phi(f)(t_i)-\phi(f)(t_i)\exp(\sqrt{-1}t h_i)\|<\ep/4
\eneq
for all $f\in {\cal F}$ and  $t\in [0,1],$  $i=1,2,...,N.$ Thus one
obtains a continuous path of unitaries $\{Z(t): t\in [t_{i-1},
t_i]\}$ with $Z(t_{i-1})=1_{M_n},$ $Z(t_i)=W_iW_{i-1}^*,$ and
\beq\label{d1uni-6+}
\|Z^*(t)\phi(f)(t_i)Z(t)-\phi(f)(t_i)\|<\ep/4\tforal f\in {\cal F},
\eneq
for $t\in [t_{i-1}, t_i],$ $i=1,2,...,N.$ One can also apply
\ref{frombk} to obtain the path $Z(t).$ Define $W(t)=Z(t)W_{i-1}$
for $t\in [t_{i-1}, t_i],$ $i=1,2,...,N.$ Note that $W\in C([0,1],
M_n).$ Moreover, for $t\in [t_{i-1}, t_i],$
\beq\label{d1uni-7}
&& \hspace{-0.8in} \|W^*(t)\phi(f)(t)W(t) -\psi(f)(t)\|
\le \|W^*(t)\phi(f)(t)W(t)-W^*(t)\phi(f)(t_{i})W(t)\|\\
&&+\|W^*(t)\phi(f)(t_{i})W(t)-\psi(f)(t_i)\|+\|\psi(f)(t_i)-\psi(f)(t)\|\\
&< & \ep_1+\|W_{i-1}^*Z(t)^*\phi(f)(t_i)Z(t) W_{i-1}-\psi(f)(t_i)\|+\ep_1\\
&<& \ep_1+\ep/4+\|W_{i-1}^*\phi(f)(t_i)W_{i-1}-\psi(f)(t_i)\|+\ep_1\\
&<&2\ep_1+\ep/4+\ep_1+\|W_{i-1}^*\phi(f)(t_{i-1})W_{i-1}-\psi(f)(t_{i-1})\|+\ep_1\\
&<& 4\ep_1+\ep/4+\ep/4<\ep
\eneq
for all $f\in {\cal F},$ $i=1,2,...,N.$

Finally, if $\phi(f)(0)=\psi(f)(0)$ for all $f\in C(X),$  we choose $W_0=1_{M_n}.$ The above proof
shows that $W(0)=1_{M_n}.$ Moreover, if $\phi(f)(1)=\psi(f)(1),$ we choose $M_N=1_{M_n}.$ The above proof also shows
that $W(1)=1_{M_n}.$
\end{proof}

\begin{lem}\label{d1local}
Let $X$ be a connected
CW complex and let $n\ge 1.$
Fix a unital \hm\, $h_0: C(X)\to M_n$ given by
$$
h_0(f)=\sum_{i=1}^m f(\xi_i)e_i\tforal f\in C(X),
$$
where $\{\xi_1, \xi_2,...,\xi_m\}$ ($m\le n$) is a subset of $m$
distinct points in $X$ and $\{e_1,e_2,...,e_m\}$ is a set of
mutually orthogonal non-zero  projections.
 For any $\ep>0,$ there
exists $\dt>0$ and a finite subset ${\cal G}\subset C(X)$ satisfying
the following: Suppose that $Y$ is a connected compact metric
space, $\phi: C(X)\to C(Y, M_n)$ is a unital \hm\, and $y_0\in Y$
for which
\beq\label{d1local-1}
\phi(f)(y_0)=h_0(f)\tforal  f\in C(X),\\\label{d1local-2}
\|\phi(g)(y)-\phi(g)(y_0)\|<\dt\tforal g\in {\cal G},\,\,\, y\in Y
\eneq
and there are continuous maps $x_i: Y\to X$ ($i=1,2,...,n$) and
mutually orthogonal rank one projections $\{q_1,q_2,...,q_n\}\subset
C(Y,M_n)$ such that
\beq\label{d1local-2+}
\phi(f)=\sum_{j=1}^n f(x_j)q_j\tforal f\in C(X).
\eneq

Then, there is
a partition $\{S_1, S_2,...,S_m\}$ of
$\{1,2,...,n\}$
such that $\xi_i=x_j(y_0)$ for some $j\in S_i,$ ${\rm dist}(\xi_i,
x_j(y))<\ep$ for all $j\in S_i,$ $\lim_{y\to y_0}
x_j(y)=\xi_i\tforal j\in S_i,$
\beq\label{d1local-3}
\|e_i-\sum_{j\in S_i}q_j(y)\|<\ep\tforal y\in Y\tand \lim_{y\to
y_0}\sum_{j\in S_i}q_j(y)=e_i,
\eneq
$i=1,2,...,m.$

\end{lem}

\begin{proof}
Let
$$
\eta_0=\min\{{\rm dist}(\xi_i, \xi_j): i\not=j, i, j\in
\{1,2,...,m\}\}.
$$

Fix any $\eta>0$ for which $\eta<\min\{\ep, \eta_0/2\}.$
From the proof of \ref{d0uniq}, there is
$\dt_1>0$ and a finite subset ${\cal G}_1\subset C(X)$ satisfying
the following: if
$$
\|\phi'(g)-h_0(g)\|<\dt_1\tforal g\in {\cal G}
$$
for any unital \hm\, $\phi' :C(X)\to M_n,$ then $\phi'$ may be written as
$$
\phi'(f)=\sum_{i=1}^n f(x_i')p_i'\tforal f\in C(X),
$$
where $\{x_1'x_2', ...,x_n'\}\subset X$ and $\{p_1', p_2,'...,p_n'\}$ is a set of mutually orthogonal rank one projections and
$
{\rm dist}(\xi_i, x_j')<\eta/2,
$
for $j\in S_i,$ where $\{S_1,S_2,...,S_m\}$ is a partition
of $\{1,2,...,n\}.$

Choose $f_i\in C(X)_+$ with $f_i(x)\le 1$ such that $f_i(x)=1$ if ${\rm
dist}(x, \xi_i)<\eta/2$ and $f_i(x)=0$ if ${\rm dist}(x,\xi_i)\ge
\eta.$  Put ${\cal G}={\cal G}_1\cup {\cal F}\cup\{1, f_i:
i=1,2,...,m\}.$ Let $\dt=\min\{\ep/4, \eta/2, \dt_1/2\}.$ Now if
$\phi: C(X)\to C(Y, M_n)$ is a unital \hm\, which satisfies
(\ref{d1local-1}) and (\ref{d1local-2}) for the above $\dt$ and
${\cal G}.$

For each $y\in Y,$  there is a partition
$S_1(y), S_2(y),...,S_m(y)$ of $\{1,2,...,n\}$
such that
\beq\label{d1local-10+}
\phi(f)(y)=\sum_{i=1}^m(\sum_{j\in S_i(y)}f(x_j(y))q_j(y))\tforal
f\in C(X),
\eneq
and ${\rm dist}(\xi_i, x_j)<\eta/2$ for all $j\in S_i(y),$
$i=1,2,...,m.$ Since $\eta<\eta_0/2,$ $i\in S_i(y)$ for
$i=1,2,...,m.$
Suppose that, for some $j\in \{1,2,...,n\},$ there are $y_1, y_2\in
Y$ with $y_1\not=y_2$ such that $j\in S_i(y_1)$ but $j\not\in
S_i(y_2).$  Then $j\in S_{i'}(y_2)$ and $i'\not=i.$

Thus
$$
{\rm dist}(\xi_i, x_j(y_1))<\eta/2\andeqn {\rm dist}(\xi_{i'},
x_j(y_2))<\eta/2.
$$
Note that ${\rm dist}(\xi_i, \xi_{i'})\ge \eta_0.$  Hence
$$
{\rm dist}(\xi_i, x_j(y_2))\ge \eta_0-\eta/2\ge
\eta_0-\eta_0/4=3\eta_0/4.
$$
Since $Y$ is path connected, there should be a point $y_3\in Y$ such that
$$
{\rm dist}(x_j(y_3),\xi_i)=\eta_0/2>\eta/2.
$$
But ${\rm dist}(x_j(y_3),\xi_l)<\eta/2$ for some $l\not=i.$ However,
$$
{\rm dist}(\xi_i, \xi_l)\le {\rm dist}(\xi_i, x_j(y_3))+{\rm
dist}(x_j(y_3), \xi_l)<\eta_0/2+\eta/2<3\eta_0/4
$$
A contradiction. Therefore, if $j\in S_i(y),$ then, for all $y\in
Y,$ $j\in S_i(y).$ This implies that $S_i(y)=S_i$ is independent of
$y.$ The above also implies that $\xi_i=x_j(y_0)$ for some $j\in
S_i.$  The continuity of $x_j(y)$ also forces
$$
\lim_{y\to y_0} x_j(y)=\xi_i
$$
for all $j\in S_i,$ $i=1,2,...,m.$

To finish the proof, one notes that, for each $y\in Y,$
\beq\label{d1local-11}
\phi(f_i)(y_0)=e_i\andeqn \phi(f_i)(y)=\sum_{j\in
S_i}q_j(y),\,\,\,i=1,2,...,m.
\eneq
Therefore
\beq\label{d1local-12}
\|e_i-\sum_{j\in S_i}q_j\|<\eta,\,\,\,i=1,2,...,m.
\eneq
Furthermore, by (\ref{d1local-11}),
\beq\label{d1local-13}
\lim_{y\to y_0}\sum_{j\in S_i}q_j(y)=\lim_{y\to y_0}\phi(f_i)(y)=
\phi(f_i)(y_0)=e_i,\,\,\, i=1,2,...,m.
\eneq

\end{proof}

\begin{df}
Let $X$ be a compact metric space, let $n\ge 1$ be an integer and
let $t\in X$ be a point. In what follows, denote by $\pi_t: C(X,
M_n)\to M_n$ the point-evaluation  \hm\, defined by $\pi_t(f)=f(t)$
for all  $f\in C(X).$
\end{df}

\begin{lem}\label{d1twopt}
Let $X$ be a locally path connected compact metric space and let
$n\ge 1$ be an integer. Then, for any $\ep>0,$ $\eta>0$ and any
finite subset ${\cal F}\subset C(X),$ there exist $\dt>0$ and a
finite subset ${\cal G}\subset C(X)$ satisfying the following: if
$\phi_1,\,\phi_2: C(X)\to M_n$ are two   unital \hm s for which
\beq\label{d1twpt-1}
\|\phi_1(g)-\phi_2(g)\|<\dt\tforal g\in {\cal G},
\eneq
then there is a unital \hm\, $\Phi: C(X)\to C([0,1], M_n)$ and there
are  continuous maps $\af_i: [0,1]\to X$ ($1\le i\le n$) and
mutually orthogonal rank one projections $\{p_1,
p_2,...,p_n\}\subset C([0,1], M_n)$ such that $\pi_0\circ
\Phi=\phi_1, \,\,\, \pi_1\circ \Phi=\phi_2,$
\beq\label{d1twopt-2}
\Phi(f)=\sum_{i=1}^n g(\af_i)p_i\tforal g\in C(X)\tand\\
\|\pi_t\circ \Phi(f)-\phi_1(f)\|<\ep\tforal f\in {\cal F}.
\eneq
Moreover,
\beq\label{d1twopt-3}
{\rm dist}(\af_i(t),\af_i(0))<\ep\tforal f\in C(X)\tand\tforal t\in
[0,1].
\eneq

\end{lem}

\begin{proof}

Let $M=\sup\{\|f\|: f\in {\cal F}\}.$ Let $\dt_1>0$ (in place of
$\dt$) be as in \ref{frombk} associated with $\ep/4$ (in place of
$\ep$) and $M.$ Let $\dt_2=\min\{\ep/4, \dt_1/2\}.$ Let $\dt>0$ and
a finite ${\cal G}\subset C(X)$ be as in \ref{d0uniqhm} associated
with $\dt_2$ (in place of $\ep$), $\eta>0$ (in place of $\ep_1$)
and ${\cal F}.$

Now suppose that (\ref{d1twpt-1}) holds for the above ${\cal G}$ and
$\dt.$ It follows from \ref{d0uniqhm} that there exist continuous
maps $\af_i': [0,1/2]\to X$ ($i=1,2,...,n$) and mutually orthogonal
rank one projections $\{e_1,e_2,...,e_n\}\subset M_n$ such that
\beq\label{d1twpt-4}
\sum_{i=1}^nf(\af_i'(0))e_i=\phi_1(f)\andeqn \|\sum_{i=1}^n
f(\af_i'(t))e_i-\phi_1(f)\|<\ep_1 \tforal f\in {\cal F}\andeqn
\eneq
for all $t\in [0, 1/2]$ and there exists a unitary $u\in M_n$ such
that
\beq\label{d1twpt-5}
{\rm ad}\, u\circ \sum_{i=1}^n f(\af_i'(1/2))e_i=\phi_2(f)\tforal
f\in C(X).
\eneq
In particular,
\beq\label{d1twpt-6}
\|u(\sum_{i=1}^n f(\af_i'(1/2))e_i)-(\sum_{i=1}^n f(\af_i'(1/2))e_i)
u\|<2\ep_1
\eneq
for all $f\in {\cal F}.$ By applying \ref{frombk}, there exists
a continuous path of unitaries $\{u(t): t\in [1/2, 1]\}\subset M_n$
such that
\beq\label{d1twpt-7}
u(1/2)=1,\,\,\, u(1)=u\andeqn \|u(t)(\sum_{i=1}^n f(\af_i'(1/2))e_i)
-(\sum_{i=1}^n f(\af_i'(1/2))e_i)u\|<\ep/4
\eneq
for all $f\in {\cal F}.$ Define $\af_i(t)=\af_i'(t)$ if $t\in
[0,1/2]$ and $\af_i(t)=\af_i'(1/2)$ if $t\in (1/2, 1],$
$i=1,2,...,n.$ Define $p_i(t)=e_i$ if $t\in (1/2,1]$ and
$p_i(t)=u(t)^*e_iu(t)$ if $t\in (1/2, 1].$ Then
\beq\label{d1twpt-8}
&&\sum_{i=1}^n f(\af_i(0))p_i(0)=\phi_1(f),\,\,\, \sum_{i=1}^n
f(\af_i(1))p_i(1)=\phi_2(f)\tforal f\in C(X)\andeqn \\
&&\|\sum_{i=1}^n f(\af_i(t))p_i(t)-\phi_1(f)\|<\ep\tforal f\in {\cal
F}.
\eneq

\end{proof}

\begin{lem}\label{d1diag}
Let $X$ be a connected finite CW complex and let $n\ge 1.$
Let $Y$ be a finite CW complex of dimension 1 and let $\phi: C(X)\to C(Y, M_n).$
Then,  for  any $\ep>0$ and any finite subset ${\cal F}\subset C(X),$ there exist
mutually orthogonal rank one rank projections $p_1, p_2,...,p_n\subset C(Y, M_n)$ and continuous maps $\af_i: Y\to X$ such that
\beq\label{d1diag-1}
\|\phi(f)-\sum_{i=1}^n f(\alpha_i)p_i\|<\ep\tforal f\in {\cal F}.
\eneq

Moreover, if $\{y_1, y_2,...,y_L\}$ is fixed, then one can also require that
\beq\label{d1diag-2}
\phi(f)(y_l)=\sum_{i=1}^n f(\af_i(y_l))p_i(y_l)\tforal f\in C(X),
\eneq
$l=1,2,...,L.$

\end{lem}

\begin{proof}
The proof of the first part follows that of 4.1 \cite{Lnhmtpr1}. Since
${\rm dim} X=1,$ one can choose a finite subset $\{\zeta_1,
\zeta_2,...,\zeta_K\}\subset X$ which is ordered in the way so that
$\zeta_{j-1} $ and $\zeta_j$ are connected by a path which is
homeomorphic to a line segment and $X$ is the union of these paths
(line segments). Without loss of generality, one may assume that
these paths are line segments and will be written as $[\zeta_{j-1},
\zeta_j],$ $j=1,2,...,K.$ Furthermore, one may assume that
$\{y_1,y_2,...,y_L\}\subset \{\zeta_1, \zeta_2,...,\zeta_K\}.$

By adding sufficiently many points to $\{\zeta_1, \zeta_2,...,\zeta_K\},$ one may also assume that
\beq\label{d1diag-4}
|f(t)-f(\zeta_{j-1})|<\ep/3\tforal f\in {\cal F}
\eneq
and all $t\in [\zeta_{j-1}, \zeta_j],$ $j=1,2,...,K.$

For each $j,$ by applying \ref{d1twopt}, there are continuous maps
$\af_{i,j}: [\zeta_{j-1}, \zeta_j]\to X$ ($i=1,2,...,n$) and
mutually orthogonal rank one projections $\{p_{1,j},
p_{2,j},...,p_{n,j}\}\subset C([\zeta_{i-1}, \zeta_j], M_n)$ such
that
\beq\label{d1diag-5}
\sum_{i=1}^n
f(\af_{i,j}(\zeta_{j-1}))p_i(\zeta_{j-1})=\pi_{\zeta_{j-1}}\circ
\phi,\,\,\, \sum_{i=1}^n f(\af_{i,j}(\zeta_j))p_i({\zeta_j})=\pi_{\zeta_j}\circ \phi\andeqn\\
\|\sum_{i=1}^n f(\af_{i,j})(t)p_i(t)-\pi_{\zeta_{j-1}}\circ
\phi(f)\|<\ep/3\tforal f\in {\cal F},
\eneq
$j=1,2,...,K.$ Define $\af_i: Y\to X$ such that
$\af_i(t)=\af_{i,j}(t)$ for $t\in [\zeta_{j-1}, \zeta_j],$ and
define $p_i(t)=p_{i,j}(t)$ for $t\in [\zeta_{j-1}, \zeta_j],$
$j=1,2,...,K.$

Thus
$$
\|\phi(f)-\sum_{i=1}^n f(\af_i)p_i\|<\ep\tforal f\in {\cal F}.
$$

\end{proof}

\section{Approximate \hm s and The Basic Homotopy Lemma}

\begin{lem}\label{d1semi}
Let  $X$ be a locally path connected compact metric space and let
$n\ge 1$ be an integer. Then, for any $\ep>0$ and any finite subset
${\cal F}\subset C(X),$ there exist $\dt>0$ and a finite subset
${\cal G}\subset C(X)$ satisfying the following: for any unital map
$\phi: C(X)\to  C([0,1],M_n)$ with $\|\phi(f)\|\le M$ for all $\|f\|\le 1$ such that
\beq\label{d1semi-1}
\|\phi(\lambda_1 x+ \lambda_2 y)-(\lambda_1\phi(x)+\lambda_2\phi(y))\|<\dt,\\\label{d1semi-2}
\|\phi(xy)-\phi(x)\phi(y)\|<\dt\tand \|\phi(x^*)-\phi(x)^*\|<\dt
\eneq
for all $\lambda_1,\lambda_2\in {\mathbb C}$ with $|\lambda_i|\le 1$ ($i=1,2$) and $x, y\in {\cal G},$ there
exists a unital \hm\, $\psi: C(X)\to C([0,1],M_n)$ such that
\beq\label{d1semi-3}
\|\phi(f)-\psi(f)\|<\ep\tforal f\in {\cal F}.
\eneq
 If, moreover, $\pi_0\circ \phi$ is a unital \hm, or both
 $\pi_0\circ \phi$ and $\pi_1\circ \phi$ are unital \hm s, then
 $\psi$ can be so chosen that $\pi_0\circ\psi=\pi_0\circ \phi$ (or
 $\pi_0\circ \psi=\pi_1\circ \phi$ and $\pi_1\circ \phi=\pi_1\circ
 \psi$).

\end{lem}

\begin{proof}
Let $\ep>0$ and a finite subset ${\cal F}\subset C(X)$ be given.
Without loss of generality, we may assume that
$\|f\|\le 1$ for all
$f\in {\cal F}.$ Let $\dt_1>0$ (in
place of $\dt$) and ${\cal G}_1$ (in place of ${\cal G}$) associated with $\ep/3,$ ${\cal F}$ and $n$ required by
\ref{d1twopt}. One may assume that ${\cal F}\subset {\cal G}_1$ and $\dt_1<\ep/2.$

Let $\dt>0$ and ${\cal G}\subset C(X)$ be a finite subset associated
with $\dt_1/3$ (in place of $\ep$) and ${\cal G}_1$ (in place of
${\cal F}$ and $M$ as required by \ref{d0semi}. One  may assume that
${\cal G}_1\subset {\cal G}.$ Suppose that $\phi$ satisfies
(\ref{d1semi-1}) and (\ref{d1semi-2}) for the above $\dt$ and ${\cal
G}.$ Let $\eta>0$ such that
\beq\label{d1semi-4}
|f(t)-f(t')|<\dt_1/3\tforal f\in  {\cal G}
\eneq
if $|t-t'|<\eta.$ Let
$$
0=t_0<t_1<\cdots <t_N=1
$$
be a partition such that $|t_i-t_{i-1}|<\eta,$ $i=1,2,...,N.$
It follows from \ref{d0semi} that, for each $i,$ there exists a
unital \hm\, $\psi_i: C(X)\to M_n$ such that
\beq\label{d1semi-4+}
\|\phi(f)(t_i)-\psi_i(f)\|<\dt_1/3 \tforal f\in {\cal F}\cup {\cal
G}_1.
\eneq
 One estimates that
\beq\label{d1semi-5}
\|\psi_i(f)-\psi_{i-1}(f)\|&\le& \|\psi_i(f)-\phi(f)(t_i)\|\\
&&+\|\phi(f)(t_i)-\phi(f)(t_{i-1})\|+\|\phi(f)(t_{i-1})-\psi_{i-1}(f)\|\\
&<&\dt_1/3+\dt_1/3+\dt_1/3=\dt_1
\eneq
for all $f\in {\cal F}\cup {\cal G}_1,$ $i=1,2,...,N.$
It follows from \ref{d1twopt} that, for each $i,$ there is a unital \hm\, $\Phi_i: C(X)\to C([t_{i-1}, t_i], M_n)$
such that
\beq\label{d1semi-6}
\pi_{t_{i-1}}\circ \Phi_i=\psi_{i-1},\,\,\, \pi_{t_i}\circ \Phi_i=\psi_i \andeqn\\
\|\pi_t\circ \Phi_i(f)-\psi_{i-1}(f)\|<\ep/3
\eneq
for all $t\in [t_{i-1}, t_i],$ $i=1,2,...,N.$
Note that it $\pi_0\circ \phi$ (or $\pi_1\circ \phi$) is a unital \hm, then one can require
that $\pi_0\circ \Phi=\pi_0\circ\phi$ (or $\pi_1\circ \Phi=\pi_1\circ \phi$).

Define $\Phi: C(X)\to C([0,1], M_n)$ by $\pi_t\circ \Phi=\pi_t\circ \Phi_i$ for $t\in [t_{i-1}, t_i],$
$i=1,2,...,N.$  It is easy to see that $\Phi$ meets the requirements.

\end{proof}

\begin{lem}\label{d1hpn}
Let $X$ be a compact metric space and let $n\ge 1$ be an integer.
For any $\ep>0$ and any finite subset ${\cal F}\subset C(X),$ there exist a finite subset ${\cal G}\subset C(X)$ and
$\dt>0$ satisfying the following:

Suppose that $\phi: C(X)\to C([0,1], M_n)$ is a unital \hm\, and $u\in C([0,1], M_n)$ such that
\beq\label{d1hpn-1}
\|u\phi(g)-\phi(g)u\|<\dt\tforal g\in {\cal G}.
\eneq
Then there exists a continuous path of unitaries $\{U(s): s\in [0,1]\}$ in $C([0,1], M_n)$ with $U(0)=u$ and $U(1)=1$ such that
\beq\label{d1hpn-2}
\|U(s)\phi(f)-\phi(f)U(s)\|<\ep\tforal f\in {\cal F}
\eneq
and for all $s\in [0,1].$
Moreover, if
\beq\label{d1hpn-2+}
u(0)(\pi_0\circ \phi(f))=(\pi_0\circ \phi(f))u(0) {\rm (}and\,\,\,
u(1)(\pi_1\circ \phi(f))=(\pi_1\circ \phi(f))u(1){\rm )}
\eneq
for all $f\in C(X),$ then one can choose $U$ so that
\beq\label{d1hpn-2++}
\hspace{-0.3in}U(s)(0)(\pi_0\circ \phi(f))=(\pi_0\circ \phi(f))U(s)(0) \,{\rm (} and\,\,\,
U(s)(1)(\pi_1\circ \phi(f))=(\pi_1\circ \phi(f))U(s)(1){\rm )}
\eneq
for all $f\in C(X).$

\end{lem}

\begin{proof}
Fix $\ep>0$ and a finite subset ${\cal F}\subset C(X).$ Put
$Y=X\times \T.$ We identify $C(Y)$ with $C(X)\otimes C(\T).$ Denote
by ${\cal F}_1=\{f\otimes 1, 1\otimes z: f\in {\cal F}\},$ where
$z\in C(\T)$ is the identity function on the unit circle. Let
$\dt>0$ and let ${\cal G}_1\subset C(Y)$ (in place of ${\cal G}$) be
the finite subset required by \ref{d1semi} for $\ep/4$ and ${\cal
F}_1$ and $Y$ (in place of $X$).

It is easy to see, by choosing smaller $\dt,$ one may assume that
${\cal G}_1=\{f\otimes 1, 1\otimes z: f\in {\cal G}\}$ for some finite subset ${\cal G}\subset C(X).$
 Assume (\ref{d1hpn-1}) holds for $\dt$ and ${\cal G}.$
It follows from \ref{d1semi} that there exists a unital \hm\, $\Phi: C(X)\otimes C(\T)\to C([0,1], M_n)$ such that
\beq\label{d1shpn-3}
\|\Phi(f\otimes 1)-\phi(f)\|<\ep/4\andeqn \|\Phi(1\otimes z)-u\|<\ep/4
\eneq
for all $f\in {\cal F}.$
It follows from \ref{d1diag} that there are continuous functions $\alpha_i: [0,1]\to X$ and $\beta_i: [0,1]\to \T$ such that
\beq\label{d1shpn-4}
\|\Phi(f\otimes 1)-\sum_{i=1}^n f(\alpha_i)p_i\|<\ep/4\andeqn
\|\Phi(1\otimes z)-\sum_{i=1}^n\beta_ip_i\|<\ep/4
\eneq
for all $f\in {\cal F},$ where $p_1, p_2,...,p_n$ are rank one projections in $C([0,1], M_n).$

Since $\beta_i\in C([0,1], \T),$ there exists a continuous path of
unitaries $\{u_i(s): s\in [1/2,1]\}\subset C([0,1], M_n)$ such that
\beq\label{d1shpn-5}
u_i(1/2)=\beta_i\andeqn u_i(1)=1,\,\,\,i=1,2,...,n.
\eneq
Define $U(s)=\sum_{i=1}^n u_i(s)p_i$ for $s\in [1/2,1].$ Note that,
for each $s\in [1/2,1],$ $U(s)\in C([0,1], M_n).$ Moreover,
\beq\label{d1shpn-5+1}
U(1/2)=\sum_{i=1}^n\beta_ip_i, \,\,\, u(1)=1\andeqn
U(s)(\sum_{i=1}^nf(\alpha_i)p_i)=(\sum_{i=1}^nf(\alpha_i)p_i)U(s)
\eneq
for all $f\in C(X)$ and $s\in [1/2,1].$ Thus , by (\ref{d1shpn-3})
and (\ref{d1shpn-4}),
\beq\label{d1shpn-5+2}
\|U(s)\phi(f)-\phi(f)U(s)\|<\ep/2\tforal f\in {\cal F}\rforal s\in
[1/2, 1].
\eneq
Since, by (\ref{d1shpn-3}) and (\ref{d1shpn-4}),
\beq\label{d1shpn-6}
\|1_{M_n}-u^*(\sum_{i=1}^n \beta_ip_i)\|<\ep/2,
\eneq
there exists a self-adjoint element $a\in C([0,1], M_n)$ such that
\beq\label{d1shpn-6+}
u^*(\sum_{i=1}^n\beta_ip_i)=\exp(ia)\andeqn \|a\|\le 2\arcsin
(\ep/4).
\eneq
Define $U(s)=u\exp(i 2sa)$ for $s\in [0,1/2].$ Then
 $\{U(s): s\in [0,1/2]\}\subset C([0,1], M_n)$ such that
\beq\label{d1shpn-7}
U(0)=u,\,\,\, U(1/2)=\sum_{i=1}^n\beta_i(1/2)p_i\andeqn\\
\|1-u^*U(s)\|<\ep/2\tforal t\in [0,1/2]
\eneq
for all $s\in [0,1/2].$ Note that $\{U(s): t\in [0,1]\}\subset
C([0,1], M_n)$ is a continuous path of unitaries and one estimates
that
\beq\label{d1shpn-8}
\|\phi(f)U(s)-U(s)\phi(f)\|<\ep\tforal f\in {\cal F}
\eneq
for all $s\in [0,1].$

This proves the first part of the lemma. To prove the last part, by
applying the last part of \ref{d1semi}, one may choose $\Phi$ so
that
\beq\label{d1shpn-9}
\pi_0\circ \Phi(f\otimes 1)=\pi_0\circ \phi(f)\andeqn \pi_0\circ
\Phi(1\otimes z)=u(0).
\eneq
Moreover, by \ref{d1diag},
\beq\label{d1shpn-10}
\pi_0\circ \Phi(f\otimes 1)=\sum_{i=1}^nf(\af_i(0))p_i(0)\andeqn
\pi_0\circ \Phi(1\otimes z)=\sum_{i=1}^n \beta_i(0)p_i(0).
\eneq
Therefore,
\beq\label{d1shpn-11}
(\pi_0\circ \phi(f))U(s)(0)=U(s)(0)(\pi_0\circ \phi(f))
\eneq
for all $f\in C(X)$ and $s\in [1/2,1].$ In this case, one also has
$a(0)=0$ in (\ref{d1shpn-6+}). Therefore $U(s)(0)=u$ for all $s\in
[0,1/2].$ Thus, in fact, (\ref{d1shpn-11}) holds for all $s\in
[0,1].$

One can also make the same arrangement for $t=1.$

\end{proof}

\section{Approximate diagonalization }

\begin{df}\label{LAR}
Let $Y$ be a compact metric space. Recall that $Y$   is a locally
absolute retract, if, for any $y\in Y$ and  any $\ep_1>0,$ there
exist $\ep_1>\ep_2>0$ and a closed neighborhood $Z$ of $y$ such that
$B(y, \ep_2)\subset Z\subset B(y, \ep_1)$ and  $Z$ is an absolute
retract.
\end{df}

\begin{lem}\label{extension}
Let $X$ be a compact metric space which is locally absolutely
retract and $n\ge 1$ be an integer. Let $\T$ be the unit circle. Let
$\ep>0$ and ${\cal F}\subset C(X)$ be a finite subset. Suppose that
$\phi: C(X)\to C(\T, M_n)$ is a unital \hm\, satisfying the
following:

{\rm (1)} $\phi(f)=\sum_{i=1}^nf(\af_i)p_i,$ where $\af_i: \T\to X$
are  continuous maps and $\{p_1,p_2,...,p_n\}\subset C(\T, M_n)$ are
mutually orthogonal rank one projections;

{\rm (2)} $\pi_{\sqrt{-1}}\circ \phi(f)=\sum_{i=1}^{m}f(x_i)e_i,$
where $\{x_1,x_2,...,x_m\}\subset X$ are distinct points, $\{e_1,
e_2,...,e_m\}$ is a set of mutually orthogonal non-zero projections,


{\rm (3)} There is a partition $\{S_1, S_2,...,S_{m}\}$ of
$\{1,2,...,n\}$  such that, for $s\in S_i,$
\beq\label{ext-2}
{\rm dist}(x_i ,\af_s(y))< \eta_1\tforal y\in X,
\eneq
$B(x_i, \eta_1)\subset Z_i\subset B(x_i, \eta_2/4)$ and $Z_i$
is a compact subset which is also absolutely retract, $i=1,2,...,m,$ where
$$
|f(x)-f(x')|<\min\{\dt_0/2, \ep/4\} \tforal f\in {\cal G},
$$
if ${\rm dist}(x, x')<\eta_2,$ and where  $\dt_0$ (in place of
$\dt$) and ${\cal G}$ associated with  $\ep/2$ (in place of $\ep$)
and ${\cal F}$  required by \ref{d1hpn};

{\rm (4)}
\beq\label{ext-2+}
\|\pi_{\sqrt{-1}}\circ \phi(g)-\pi_t\circ\phi(g)\|<\dt_0/2\tforal
g\in {\cal G}\tand t\in \T.
\eneq

Then, there exist continuous maps $\gamma_i: D\to X$ and mutually
orthogonal rank one projections $\{q_1,q_2,...,q_n\}\subset C(D,
M_n)$ such that, for any $t\in \T$ and any $y\in D,$
\beq\label{ext-3}
\|\pi_t\circ \phi(f)-\sum_{i=1}^n f(\gamma_i(y))q_i(y)\|<\ep\tforal
f\in {\cal F},\\
\label{ext-4} \gamma_i|_{\T}=\alpha_i\tand q_i|_\T=p_i,
\eneq
$i=1,2,...,n,$ where $D$ is the unit disk.

\end{lem}

\begin{proof}
 Let $C_+$ be the unit upper semi-circle and $C_-$ be the unit
lower semi-circle. Let $L_+=\{1+a\sqrt{-1}: -1\le a\le 0\},$
$L_-=\{-1+a\sqrt{-1}: -1\le a\le 0\}$ and let $C_+'=\{z-\sqrt{-1}:z\in
C_+\}.$ Denote by $\Omega$ the compact subset of the plane with boundary
consisting of $C_+, C_+',L_+$ and $L_-.$

For each $i$ ($1\le i\le n$), define $\beta_i: \partial \Omega \to
X$ as follows:
$$
\beta_i(y)=\begin{cases} \alpha_i(y), & \text{if $y\in C_+$;}\\
                                        \alpha_i(-1) &\text{if $y\in L_-$;}\\
                                        \alpha_i(1) &\text{if $y\in L_+$} \andeqn\\
                                        \alpha_i(e^{-\sqrt{-1}\theta}) &\text{if $y=e^{\sqrt{-1} \theta}-\sqrt{-1}$.} \end{cases}
 $$

Note that, by (\ref{ext-2}),  $\beta_i(y)\in Z_i.$  Since $Z_i$ is
an absolute retract, there is a continuous map ${\bar \beta}_i:
\Omega\to Z_i$ such that ${\bar \beta}_i|_{\partial
\Omega}=\beta_i,$ $i=1,2,...,n,$ Define $\Phi_1: C(X)\to C(\Omega,
M_n)$ by
\beq\label{ext-10+1}
\Phi_1(f)(y)=\sum_{i=1}^n
f({\bar\beta}_i(y))p_i(e^{\sqrt{-1}\theta})
\eneq
for $y=e^{\sqrt{-1}\theta}-a\sqrt{-1}$ for some $-1\le a\le
0$ and $0\le \theta\le \pi,$ and for all $f\in C(X).$  By (3),
\beq\label{ext-10+2}
\|\Phi(f)(y)-\phi(f)(\sqrt{-1})\|<\min\{\dt_0/2, \ep/4\} \rforal
y\in \Omega.
\eneq
Since $p_i$ has rank one, there exists a unitary $u\in C(C_+, M_n)$
such that
\beq\label{ext-10}
u(e^{\sqrt{-1}
t})^*p_i(e^{\sqrt{-1}t})u(e^{\sqrt{-1}t})=p_i(e^{-\sqrt{-1}t})\rforal
t\in [0, \pi],
\eneq
$i=1,2,...,n.$
Define $\phi_1, \phi_2: C(X)\to C(C_+, M_n)$ by
\beq\label{ext-11}
\phi_1(f)(e^{\sqrt{-1}\theta})&=&\sum_{i=1}^n
f(\af_i(e^{-\sqrt{-1}\theta}))p_i(e^{\sqrt{-1}\theta})
\andeqn\\\label{ext-12}
\phi_2(f)(e^{\sqrt{-1}\theta})&=&\sum_{i=1}^nf(\alpha_i(e^{-\sqrt{-1}\theta}))p_i(e^{-\sqrt{-1}\theta})
\eneq
for all $0\le \theta\le \pi.$ Then,
\beq\label{ext-13}
u(e^{\sqrt{-1}\theta})^*\phi_1(f)(e^{\sqrt{-1}\theta})u(e^{\sqrt{-1}\theta})=\phi_2(f)(e^{\sqrt{-1}\theta})
\eneq
for all $0\le \theta\le \pi.$ By (3),
\beq\label{ext-14}
\max_{0\le \theta \le
\pi}\|\phi_1(f)(e^{\sqrt{-1}\theta})-\phi(f)(e^{\sqrt{-1}\theta})\|<\min\{\dt_0/2,
\ep/4\}.
\eneq
By (4), one has that
\beq\label{ext-15}
\|\phi_2(f)-\phi(f)\|_{C_+}<\dt_0/2 \tforal f\in {\cal
G}.
\eneq
Note that
\beq\label{ext-16}
\phi_1(f)(-1)=\phi_2(f)(-1)\andeqn \phi_1(f)(1)=\phi_2(f)(1)
\eneq
for all $f\in C(X).$ It follows from \ref{d1hpn} that there exists a
continuous path of unitaries  $\{U_s: s\in [0,1]\}\subset C(C_+,
M_n)$ such that $U_1=u,\,\,\, U_0=1_{M_n}$ and, for all $s\in
[0,1],$
\beq\label{ext-17}
\|U_s^*\phi_1(f)U_s-\phi_1\|_{C_+}<\ep/2\tforal f\in {\cal F}.
\eneq
Moreover,
\beq\label{ext-18}
U_s(-1)^*\phi_1(f)(-1)U_s(-1)=\phi_1(f)(-1)\andeqn
U_s(1)^*\phi_1(f)(1)U_s(1)=\phi_1(f)(1)
\eneq
for all $f\in C(X).$

Let $C_-'=\{z-2\sqrt{-1}: z\in C_-\},$ $L_-'=\{-1+b\sqrt{-1}: -2\le
b\le -1\}$ and $L_+'=\{1+b\sqrt{-1}: -2\le b\le -1\}.$ Denote by
$\Omega_1$ the compact subset with the boundary consisting of
$C_+',L_-', L_+'$ and $C_-'.$ Define $\Phi_2: C(X)\to C(\Omega_1,
M_n)$ as follows: If
$y=(1-s)(e^{\sqrt{-1}\theta})-\sqrt{-1})+s(e^{-\sqrt{-1}\theta}-2\sqrt{-1}),$
then define
\beq\label{ext-19}
\Phi_2(f)(y)=U_s(e^{\sqrt{-1}\theta})^*\phi_1(f)(e^{\sqrt{-1}\theta})U_s(e^{\sqrt{-1}\theta})
\eneq
for all $s\in [0,1]$ and $0\le \theta\le \pi.$ Then
\beq\label{ext-19+}
\Phi_2(f)(y)=\sum_{i=1}^n
f(\af_i(e^{-\sqrt{-1}\theta}))(U_s^*p_iU_s)(e^{\sqrt{-1}\theta})
\eneq
for all $f\in C(X),$ if
$y=(1-s)(e^{\sqrt{-1}\theta})-\sqrt{-1})+s(e^{-\sqrt{-1}\theta}-2\sqrt{-1}),$ $s\in
[0,1]$ and $0\le \theta\le \pi.$

Note that
\beq\label{ext-20}
\Phi_2(f)(e^{\sqrt{-1}\theta}-\sqrt{-1})=\Phi_1(f)(e^{\sqrt{-1}\theta}-\sqrt{-1}),\\\label{ext-21}
\Phi_2(f)(e^{\sqrt{-1}\theta}-2\sqrt{-1})=\phi(f)(e^{-\sqrt{-1}\theta})\\\label{ext-22}
\Phi_2(f)(y)=\phi(f)(-1)\andeqn \Phi_2(f)(y')=\phi(f)(1)
\eneq
for any $y\in L_-',$ $y'\in L_+'$ and $0\le \theta\le \pi.$ Let
$\Omega_2=\Omega\cup \Omega_1.$ By (\ref{ext-20}), one can define
$\Phi: C(X)\to C(\Omega_2, M_n)$ by
\beq\label{ext-23}
\Phi(f)(y)=\Phi_1(f)(y)\,\,\,{\rm if}\,\,\, y\in \Omega\andeqn
\Phi(f)(y)=\Phi_2(f)(y)\,\,\,{\rm if}\,\,\, y\in \Omega_1.
\eneq
Fix $3/4<d_0<1.$ Let
$$
S_-=\{-d_0+(1-d_0)e^{i\theta}: |\theta|\le \pi\}\andeqn
S_+=\{d_0+(1-d_0)e^{i\theta}: |\theta|\le \pi\}
$$
be two small circles with the centers $-d_0$ and $d_0,$ respectively. Let $A$ be the connected subset containing the
origin and bounded by $S_-, S_+, C_+$ and $C_-.$ Denote by $A^{o}$
the interior of $A.$ Then there exists a continuous map $\Gamma:
\Omega_2\to A$ which is a homeomorphism from the interior of
$\Omega_2$ onto $A^o,$ which fixes $C_+,$ maps $C_-'$ onto $C_-$
such that $\Gamma(e^{\sqrt{-1}\theta}-2\sqrt{-1})=e^{\sqrt{-1}\theta}$ and
maps $L_-\cup L_-'$ onto $S_-$ and $L_+\cup L_+'$ onto $S_+.$

Now define $\psi: C(X)\to C(D, M_n)$ by
\beq\label{ext-24}
\psi(f)(y)=\begin{cases} \Phi(f(\Gamma^{-1}(y))), &\text{if $y\in A^o$;}\\
                          \phi(f)(y), & \text{if $y\in C_+$;}\\
                          \phi(f)(y), &\text{if $y\in C_-$;}\\
                          \phi(f)(-1), &\text{if $|y+d_0|\le 1-d_0$;}\\
                          \phi(f)(1), &\text{if $|y-d_0|\le 1-d_0$}
                          \end{cases}
\eneq
for all $f\in C(X).$ That $\psi$ maps $C(X)$ into $C(D, M_n)$
follows from (\ref{ext-10+1}),  (\ref{ext-20}), (\ref{ext-21}) and
(\ref{ext-22}). By (\ref{ext-10+1}) and (\ref{ext-19+}), there are
continuous maps $\gamma_i: D\to X$ and mutually orthogonal rank one
projections $q_1,q_2,...,q_n\in C(D, M_n)$ such that
\beq\label{ext-25}
\psi(f)(y)=\sum_{i=1}^n f(\gamma_i(y))q_i(y)\tforal f\in C(X).
\eneq
Moreover,
\beq\label{ext-26}
\gamma_i|_{\T}=\af_i\andeqn q_i|_{\T}=p_i,\,\,\,i=1,2,...,n
\eneq
It follows (\ref{ext-10+2}) and (\ref{ext-15})
\beq\label{ext-27}
\|\pi_t\circ \phi(f)-\sum_{i=1}^n f(\gamma_i(y))q_i(y)\|<\ep\tforal
f\in {\cal F}.
\eneq

\end{proof}

\begin{thm}\label{d2diag}
Let $X$ be a compact metric space which is a locally absolute retract and let $n\ge 1.$ Suppose that $Y$ is a compact metric space
with ${\rm dim}Y\le 2$   and suppose that $\phi: C(X)\to C(Y, M_n)$ is a unital \hm.
Then, for any $\ep>0$ and any finite subset ${\cal F}\subset C(X),$ there exist continuous
maps $\af_i: Y\to X$ ($1\le i\le n$) and mutually orthogonal rank one projections $e_1,e_2,...,e_n\in C(Y, M_n)$ such that
\beq\label{d2diag-1}
\|\phi(f)-\sum_{i=1}^n f(\af_i)e_i\|<\ep\tforal f\in {\cal F}.
\eneq
\end{thm}

\begin{proof}
Fix $\ep>0$ and a finite subset ${\cal F}\subset C(X).$ Let
$\dt_0>0$ (in pace of $\dt$) and ${\cal G}_1\subset C(X)$ (in place
of ${\cal G}$) be a finite subset associated with $\ep/16$ and
${\cal F}$ required by \ref{d1hpn} (for the given $n$).  One may
assume that ${\cal F}\subset {\cal G}_1.$ Let $\eta>0$ be such that
\beq\label{d2diag-2}
|f(x)-f(x')|<\min\{\dt_0/2,\ep/16\}\tforal f\in {\cal G}_1,
\eneq
provided that ${\rm dist}(x, x')<\eta.$ Since $X$ is locally
absolute retract and compact, there exists $\eta_1>0$ such that, for
any $x\in X,$ $B(x, \eta_1)\subset Z_x\subset B(x, \eta/2),$ where
$Z_x$ is a compact subset which  is also an absolute retract. For
each $y\in Y,$ let $\dt_1(y)>0$ (in place of $\dt$) and ${\cal
G}_2'(y)\subset C(X)$ (in place of ${\cal G}$) be a finite subset
associated with $\min\{\eta_0/3, \ep/16\}$ (in place of $\ep$) and
$\pi_y\circ \phi$ required by \ref{d1local}.

For each $y\in Y,$ let $\dt_2(y)>0$ (in place of $\dt$) and ${\cal G}_2(y)\subset C(X)$
(in place of ${\cal G}$) be a finite subset associated with $\min\{\eta_0/3, \ep/6\}$
(in place of $\ep$) and $\pi_y\circ \phi$ (in place of $\phi_1$) required
by \ref{d1twopt}. Without loss of generality, to simplify notation, we may assume
that $\dt_1(y)\le \dt_2(y)$ and
${\cal G}_1\cup {\cal G}_2'(y)\subset {\cal G}_2(y).$


For each $y,$ there exists $d'(y)>0$ such that
\beq\label{d2diag-3}
\|\pi_y\circ\phi(g)-\pi_{y'}\circ\phi(g)\|<\min\{\dt_0/2,\dt_2(y)/3, \ep/16\}\tforal g\in {\cal G}_2(y),
\eneq
provided that ${\rm dist}(y,y')<d'(y).$

Fix $r>0.$ For each $y\in Y,$ let $d(y)=d'(y)r.$
 Now $\cup_{y\in Y} B(y, d(y)/12)\supset Y.$ Let
$y_1, y_2,...,y_K\in Y$ be a finite subset such that $\{B(y_i,
d(y_i)/12): i=1,2,...,K\}$ covers $Y.$ Moreover, one may assume that
the order of the cover $\le 2.$ One builds a simplicial complex as
follows: $y_1,y_2,...,y_N$ are vertices and $0$-simplexes, and
$y_{i_1}y_{i_2}$ or $y_{i_1}y_{i_2}y_{i_3}$ form a $1$-simplex (or
$2$-simplex) if and only if
\beq\label{d2diag-4}
&&B(y_{i_1}, d(y_{i_1})/12)\cap B(y_{i_2},
d(y_{i_2})/12)\not=\emptyset\\\label{d2diag-4-} && {\rm (}\andeqn
\,B(y_{i_1} d(y_{i_1})/12)\cap B(y_{i_2}, d(y_{i_2})/12)\cap
B(y_{i_3}, d(y_{i_3}/12))\not=\emptyset. {\rm )}
\eneq

Denote by ${\cal S}(r)$ the  simplicial complexes constructed this
way and by $S(r)$ the underline polyhedra. Moreover, if $y_iy_j$  is
a $1$-simplex, then
\beq\label{d2diang-4+}
{\rm dist}(y_i, y_j)<\max\{d(y_i)/6, d(y_j)/6\}.
\eneq

If $y_j$ is a vertex, then there are points $\af_j^{(k)}(y_j)\in X,$ $k=1,2,...,n,$
and mutually orthogonal rank one projections
$p^{j}_1, p^{j}_2,...,p^{j}_n\in M_n$ such that
\beq\label{d2diag-4+1}
\phi(f)(y_j)=\sum_{k=1}^n f(\af_j^{(k)}(y_j))p_k^j\rforal f\in C(X).
\eneq

Denote by $I_{i,j}$ the line segment defined by $y_iy_j.$
Therefore, by applying \ref{d1twopt}, there are continuous maps
$\af_{i,j}^{(k)}: I_{i,j}\to X,$ $k=1,2,...,n,$ and mutually
orthogonal rank one projections $p_1^{i,j},
p_2^{i,j},...,p_n^{i,j}\in C(I_{i,j}, M_n)$ such that
\beq\label{d2diag-5}
\sum_{k=1}^n f(\af_{i,j}^{(k)}(y_s))p_k^{i,j}(y_s)=\pi_{y_s}\circ \phi(f)\tforal f\in C(X),\\
\|\pi_{y_{s'}}\circ
\phi(g)-\sum_{k=1}^nf(\af_{i,j}^{(k)}(t))p_k^{i,j}(t)\|<\min\{\dt_2(y_{s'})/2,
\ep/16\} \tforal g\in {\cal G}_2(y_{s'}),
\eneq
where $s=i,j$ and $s'=i,$  or $j$ if $\max \{d(y_i)/6,
d(y_j)/6\}=d(y_i)/6,$ or $\max \{d(y_i)/6, d(y_j)/6\}=d(y_j)/6.$

Let $I(r)=\cup I_{i,j}$ be the union of  all $0$-simplex and $1$-simplexes in $S(r).$ One obtains a
unital \hm\, $\Phi': C(X)\to C(I(r), M_n)$ defined by
\beq\label{d2diag-6}
\pi_t\circ \Phi'(f)=\sum_{k=1}^n f(\af_{i,j}^{(k)}(t))p_k^{i,j}(t)
\eneq
if $t\in I_{i,j},$ and
\beq\label{d2diag-6+}
\pi_{y_j}\circ \Phi'=\pi_{y_j}\circ \phi.
\eneq
Define  $\af_k': I(r)\to X$ by $\af_k'(t)=\af_{i,j}^{(k)}(t)$ if $t\in I_{i,j}$ and define
projections $p_i', p_2',...,p_n'$ in $C(I(r), M_n)$ by $p_k'(t)=p_k^{i,j}(t)$ if $t\in I_{i,j}.$
Next one extends $\pi_t\circ \Phi'$ on $S(r).$

To do this, one assumes that $y_{i_1}y_{i_2}y_{i_3}$ is a 2-simplex.
Then
\beq\label{d2diag-6++}
{\rm dist}(y_{i_j}, y_{i_{j'}})<d(y_{i_j})/6 \rforal j=1,2,3
\eneq
and for one of some $j'\in \{1,2,3\}.$ Without loss of generality,
one may assume that $3=j'.$

One identifies the $2$-polyhedron $K_{i_1,i_2, i_3}$ determined  by
$y_{i_1}y_{i_2}y_{i_3}$  with the unit disk and identifies $y_{i_1}$
with $1,$ $y_{i_2}$ with $-1$ and $y_{i_3}$ with $\sqrt{-1}.$  Here
the line segments determined by $y_{i_1}y_{i_2},$ $y_{i_1}y_{i_3}$
and $y_{i_2}y_{i_3}$ with the arc with end points $-1$ and $1,$ the
arc with end points $1$ and $\sqrt{-1},$ and the arc with end points
$\sqrt{-1}$ and $-1,$ respectively.

Let $\Psi$ be the restriction of $\Phi'$ on the unit circle $\T$
(with the above mentioned identification). Then it is clear that
$\Psi$ satisfies (1), (2) and (4) in \ref{extension} (by replacing
$\phi$ by $\Psi$) for $\ep/4$ (in place of $\ep$) and ${\cal F}. $
By the choice of each $\dt_1(y)$ and by \ref{d1local}, (3) is also
satisfied (for $\Psi$). By applying \ref{extension}, and identifying
the unit dick  $D$ with $K_{i_1,i_2, i_3}$, one obtains a unital
\hm\, $\Phi_{i_1,i_2, i_3}: C(X)\to C(K_{i_1,i_2, i_3}, M_n),$
continuous maps $\af_{i_1, i_2, i_3}^{(k)}: K_{i_1, i_2, i_3}\to X$
and mutually orthogonal rank one projections $\{p_k^{i_1,i_2,i_3}:
k=1,2,...,n\}\subset C(K_{i_1, i_2, i_3}, M_n)$ such that (where
$K_{i_j, i_{j'}}$ is the 1-simplex determined by
$y_{i_j}y_{i_{j'}}$)
\beq\label{d2diag-7}
\Phi_{i_1,i_2, i_3}(f)=\sum_{k=1}^n f(\af_{i_1,i_2,i_3}^{(k)})p_k^{i_1,i_2, i_3}\tforal f\in C(X),\\
\Phi_{i_1,i_2, i_3}(f)|_{K_{i_j,i_{j'}}}=\sum_{k=1}^n f(\af_{i_j,i_{j'}}^{(k)})p_k^{i_j,i_{j'}}\tforal f\in C(X)\andeqn\\
\|\Phi_{i_1,i_2, i_3}(f)(t)-\sum_{k=1}^n f(\af_{i_1,i_2,i_3}^{(k)})(s)p_k^{i_1,i_2, i_3}(s)\|<\ep/2
\eneq
for all $t$ in the boundary of $K_{i_1, i_2, i_3},$ $s\in K_{i_1,
i_2, i_3}$ and for all $f\in {\cal F}.$ Define $\af_k: Y\to S(r)$ by
$\af_k(y_j)=y_j,$ $\af_k(y)=\af_{i,j}^{(k)}(y)$ if $y$ is in the
polyhedron determined by $y_iy_j$ and $\af_k(y)=\af_{i_1,i_2,
i_3}^{(k)}(y)$ if $y\in K_{i_1, i_2, i_3}.$ Define $p_k\in C(Y,
M_n)$ by $p_k(y_j)=p_k^j,$ $p_k(y)=p_k^{i,j}(y)$ if $y\in K_{i,j}$
and $p_k(y)=p_k^{i_1,i_2, i_3}(y)$ if $y\in K_{i_1,i_2,i_3}.$ Define
$\psi: C(X)\to C(S(r), M_n)$ by
\beq\label{d2diag-8-1}
\psi(f)=\sum_{k=1}^n f(\af_k)p_k\tforal f\in C(X).
\eneq
Note that
$\psi(f)(t)=\Phi_{i_1, i_2, i_3}(f)(t)$ if
$t\in K_{i_1,i_2, i_3}$ and $\psi(f)(t)=\sum_{k=1}^n f(\af_{i,j})(t)p_k^{i,j}(t)$ if $t\in K_{i,j}.$
Moreover,
\beq\label{d2diag-8}
\psi(f)(y_j)=\phi(f)(y_j)\tforal f\in C(X),\,\,\,j=1,2,...,K, \andeqn\\\label{d2diag-8+}
\|\psi(f)(y)-\psi(f)(y_j)\|<\ep/4\tforal f\in C(X)
\eneq
and for some $j$ so that $y$ is in a simplex with $y_j$ as one of the vertex.

Now  one changes $r.$  To simplify notation, one may assume that
${\rm diam}(Y)\le 1.$  One obtains a sequence of open covers ${\cal
U}_m=\{B_j^{(m)}\}=\{B(y_j^{(m)}, d(y_j^{(m)},r_m)/12):
j=1,2,...,K(m)\}$ (with $d(y, r_m)=d'(y)r_m$) such that:
 (i) the
order of the cover is at most $2,$ and,

(ii)
\beq\label{d2diag++}
r_{m+1}<\min\{\ep_m/2,
\min\{d'(y_j^{(m)})/2^{m+1}: 1\le j\le K(m)\}\},
\eneq
 where $\ep_m$ is
a Lebesque number for the cover ${\cal U}_m.$ It follows from (ii)
that (iii) holds:  if $B_{j_1}^{(m+1)}\cap B_{j_2}^{(m+1)}\cap\cdots
B_{j_l}^{(m+1)}\not=\emptyset,$ then there exists $k\le K(k)$ such
that $B_{j_s}^{(m+1)}\subset B_k^{(m)},$ $s=1,2,...,l.$ For each
$m=1,2,...,$ let ${\cal S}_m$ be the simplicial complex  constructed
from points $\{y_1, y_2,...,y_{K(m)}\}$ as above, and let $S(r_m)$
be the underline polyhedra of dimension at most 2 (see
(\ref{d2diag-4}) and (\ref{d2diag-4-})). Denote by $\psi_m: C(X)\to
C(S(r_m), M_n)$ the unital \hm\, constructed above using $r=r_m,$
$m=1,2,....$

To specify the map $\pi_m^{m+1}: S(r_{m+1})\to S(r_m),$ for each $j$
($\le K(m+1)$), let $y_j^{(m+1)}$ be one of the vertex. By virtue of
(iii) above, the family
$$
{\cal U}_{j,m}=\{ B_{j'}^{(m)}: B_j^{(m)}\subset B_{j'}^{(m)}\}
$$
is non-empty. Since $\cap { B_{j'}^{(m)}\in {\cal U}_{j,m}}\not=\emptyset,$ the vertices of $S_m$ which correspond to the members of ${\cal U}_{m,j}$ span a simplex $K^{(j,m)}\in {\cal S}_m.$ Define
\beq\label{d2diag-9}
\pi_m^{m+1}(y_j^{(m+1)})=b(K^{(j,m)}),
\eneq
where $b(K^{(j,m)})$ denotes the barycenter of $K^{(j,m)}.$
As in the proof 1.13.2 of \cite{RE}, this implies that, for every simplex $S\in {\cal S}(r_{m+1}),$
the images of vertices of $S$ under $\pi_m^{m+1}$ are contained in a simplex $T\in {\cal S}(r_m).$

This construction leads to an inverse limits $\lim_{\leftarrow m}
(S_{m+1}, \pi_m^{m+1})$ which is homeomorphic to $Y$ (see the proof
of 1.13.2 of \cite{RE}).  One identifies these two spaces. Denote by
$\pi_m^{\infty}: Y\to S_m$ the continuous map induced by the inverse
limit.

Denote by $J_m: C(S(r_m), M_n)\to C(S(r_{m+1}), M_n)$ the unital
\hm\, defined by
\beq\label{d2diag-10}
J_m(f)(y)=f(\pi_m^{m+1}(y))\tforal f\in C(S(r_m), M_n),
\eneq
$m=1,2,....$ Denote by $J_{m, \infty}: C(Y, M_n)\to C(S(r_m), M_n)$
the unital \hm\, induced by the inductive limit $C(Y,
M_n)=\lim_{m\to\infty}(C(S(r_m), M_n), J_m)$ which can also be
defined by $J_{m, \infty}(f)(y)=f(\pi_m^{\infty}(y))$ for all $f\in
C(Y, M_n).$

Fix $y\in Y$ and $m.$ There is a simplex $K_m\in {\cal S}(r_m)$ such that
$\pi_m^{\infty}(y)\in K_m$ and therefore there exists a vertex $y_{j(m)}^{(m)}$ such that
\beq\label{d2diag-11}
{\rm dist}(y, y_{j(m)}^{(m)})<d'(y_{j(m)}^{(m)})/6\cdot 2^m
\eneq
(see for example the proof of 1.13.2 of \cite{RE}).
Let $\psi_m: C(X)\to S(r_m)$ be the unital \hm\, construct above (by replacing $r$ by $r_m$).
So
$$
\psi_m(f)=\sum_{k=1}^n f(\af_k)p_k\rforal f\in C(X),
$$
where $\af_k$ and $p_k$ ($k=1,2,...,n$) as constructed above (with $r$ replaced by $r_m$).

One estimates, by (\ref{d2diag-11}), (\ref{d2diag-3}), (\ref{d2diag-8}) and (\ref{d2diag-8+}), that
\beq\nonumber
\|\phi(f)(y)-J_{m, \infty}\circ \psi_m(f)(y)\| &\le &
\|\phi(f)(y)-\phi(f)(y_{j(m)}^{(m)})\|+\|\phi(f)(y_{j(m)}^{(m)})-\psi_m(f)(y_{j(m)}^{(m)})\|\\\nonumber
&&+\|\psi_m(f)(y_{j(m)}^{(m)})-\psi_m(f)(\pi_m^{\infty}(y))\|\\\label{d2diag-12}
&<& \ep/16+0+\ep/4<\ep
\eneq
for all $f\in {\cal F}.$ Note that
\beq\label{2ddiag-13}
J_{m, \infty}\circ \psi_m(f)=\sum_{k=1}^n f(\af_k\circ \pi_m^{\infty})p_k
\eneq
for all $f\in C(X).$
This completes the proof.

\end{proof}

\begin{df}
Let $Y$ be a compact metric space  and $C\subset C(Y, M_n)$ be a
unital \SCA. $C$ is said to be {\it diagonalized }if there are mutually
orthogonal rank one projections $\{p_1, p_2,..,p_n\}\subset C(Y,
M_n)$ such that $p_i$ commutes with every element in $C,$ $i=1,2,...,n.$

\end{df}

\begin{thm}\label{M1}
Let $X$ be a compact metric space and let $n\ge 1.$ Suppose that $Y$
is a compact metric space with ${\rm dim}\, Y\le 2$ and suppose that
$\phi: C(X)\to C(Y, M_n)$ is a unital \hm. Then, for any $\ep>0$ and
any compact subset ${\cal F}\subset C(X),$ there is a unital
commutative \SCA\, $B\subset C(Y, M_n)$ which can be diagonalized
and
\beq\label{M1-1}
{\rm dist}(\phi(f), B)<\ep\tforal f\in {\cal F}.
\eneq

\end{thm}

\begin{proof}
One may view that $X\subset I,$ where $I$ is the Hilbert cube which
is viewed as a subset of $l^2.$ Note that  Hilbert cube is convex
and every open (or closed)  ball in $I$ is convex and therefore is
absolute retract and locally absolute retract. Thus
$X=\cap_{m=1}^{\infty} F_m,$ where each $F_m$ is a finite union of
closed balls of $I.$ In particular, each $F_m$ is locally absolute
retract. Let $\imath_m: F_{m+1}\to F_m$
($m=1,2,...$) and $\imath_m^{\infty}: X\to F_m$ be the embeddings
Let $j_m: C(F_{m})\to C(F_{m+1})$ defined by $j_m(f)=f(\imath_m)$
and $j_{m, \infty}: C(F_m)\to C(X)$ by
$j_m^{\infty}(f)=f(\imath_m^{\infty})$ for all $f\in C(X).$ Now let
$\ep>0$ and a finite subset ${\cal F}\subset C(X)$ be given. For
each $f\in {\cal F},$ there is $m\ge 1$ and $g_f\in C(F_m)$ such
that
\beq\label{M1-2}
\|j_m^{\infty}(g_f)-f\|<\ep/2\tforal f\in {\cal F}.
\eneq
Let ${\cal G}=\{g_f: f\in {\cal F}\}.$  By considering
$\phi_m=\phi\circ j_m^{\infty},$ one obtains, by applying
\ref{d2diag}, continuous maps $\beta_i: Y\to F_m$ and mutually
orthogonal rank one projections $p_1, p_2,...,p_n\in C(Y, M_n)$ such
that
\beq\label{M1-3}
\|\phi_m(g_f)-\sum_{i=1}^n g_f(\beta_i)p_i\|<\ep/2\tforal f\in {\cal
F}.
\eneq
By combing (\ref{M1-3}) with (\ref{M1-2}), one has
that
$$
\|\phi(f)-\sum_{i=1}^n g_f(\beta_i)p_i\|<\ep
\rforal f\in {\cal F}.
$$
Let $B$ be the commutative \SCA\, generated by $\sum_{i=1}^n
g_f(\beta_i)p_i$ for $f\in {\cal F}.$
\end{proof}

\begin{cor}\label{C1}
Let $Y$ be a compact metric space with ${\rm dim} Y\le 2,$ let $n\ge
1$ be an integer and let $x$ be a normal element. Then, there are $n$
sequences of functions $\{\lambda_k^{(m)}\}$ in $C(Y)$
($k=1,2,...,n$) and there is a sequence of sets of $n$ mutually
orthogonal rank one projections $\{p_1^{(m)},
p_2^{(m)},...,p_n^{(m)}\}\subset C(Y, M_n)$ such that
$$
\lim_{m\to\infty} \|x-\sum_{k=1}^n \lambda_k^{(m)}p_k^{(m)}\|=0.
$$
Moreover, if $x$ is self-adjoint, $\lambda_k^{(m)}$ can be chosen to
be real and if $x$ is a unitary, $\lambda_k^{(m)}$ can be chosen so
that $|\lambda_k^{(m)}|=1,$ $k=1,2,...,n$ and $m=1,2,....$
\end{cor}

\section{Higher dimensional cases}
 In this section, we consider the cases that ${\rm dim} Y\ge 3.$ One
 would hope that the similar argument used in section 6 can repeat
 for higher dimensional space $Y.$ In fact, a version of
 \ref{d1semi} and \ref{d1hpn} can be proved for two dimensional
 spaces. However, the last requests in \ref{d1semi} and \ref{d1hpn}
 can not be improved, for example, in a generalized \ref{d1hpn}, $U(s)$ can not be
 chosen so it exactly commutes with $\phi$ on a given line segment even
 $u$ can. The
 reason is that not every \hm\, to $C([0,1], M_n)$ can be exactly
 diagonalized (see \cite{GP}). This technical problem is fatal as
 one can see from the results of this section. Nevertheless, one has the
 following.

\begin{prop}\label{zerodim}
Let $X$ be a zero dimensional compact metric space, $n\ge 1$ and let
$Y$ be a compact metric space for which every minimal projection in
$C(Y, M_n)$ has rank one or zero at each point of $Y$ (which is the
case if ${\rm dim} Y\le 3$---see \ref{rem7} ).

Then any unital \hm\, $\phi: C(X)\to C(Y, M_n)$ can be approximately
diagonlaized.

\end{prop}

\begin{proof}
We may assume that $X\subset \R.$ Choose an element $x\in C(X)$ with
$sp(x)=X.$ If $f\in C(X)$ then one has $f=f(x).$ For any $\dt>0,$
there are mutually orthogonal projections
$\{e_1,e_2,...,e_N\}\subset C(X)$ for which $\sum_{j=1}^N e_j=1$ and
real numbers $\lambda_1, \lambda_2,..., \lambda_N\in X$ such that
\beq\label{zerodim-1}
\|x-\sum_{j=1}^N \lambda_j e_j\|<\dt.
\eneq
Given a finite subset ${\cal F}\subset C(X)$ and $\ep>0,$ by
choosing a sufficiently small $\dt,$ one may assume that
\beq\label{zerodim-2}
\|f(x)-\sum_{j=1}^Nf(\lambda_j)e_j\|<\ep\rforal f\in {\cal F}.
\eneq
Let $\phi: C(X)\to C(Y, M_n)$ be a unital \hm. Then
\beq\label{zerodim-3}
\|\phi(f)-\sum_{j=1}^Nf(\lambda_j)\phi(e_j)\|<\ep\rforal f\in {\cal
F}.
\eneq
Since each $\phi(e_j)$ is a projection, by the assumption, there are
mutually orthogonal projections $\{p_{j,1},
p_{j,2},...,p_{j,r(j)}\}\subset C(Y, M_n)$ such that $p_{j,i}(y)$
has rank either one or zero and
$\sum_{i=1}^{r(j)}p_{j,i}=\phi(e_j),$ $j=1,2,...,N.$ Since
$\sum_{j=1}^N \phi(e_j)=1_{C(Y, M_n)},$ it is easy to find mutually
orthogonal rank one projections $p_1, p_2,...,p_n,$ where each $p_i$
is a sum of some projections $p_{j,k},$ such that $\sum_{i=1}^n
p_i=1_{C(Y, M_n)}.$  Suppose that $p_i=\sum_k p_{j_k, i_k},$
$n=1,2,...,n.$  Let $Y_{j_k, i_k}$ be the clopen set so that
$p_{j_k,i_k}(y)\not=0.$ Define a continuous map $\af_i: Y\to X$ by
$\af_i(y)=\lambda_{j_k}$ if $y\in Y_{j_k, i_k}$ and $\af_i(y)=0$ if
$y$ is not in the support if $p_i.$ Then, by (\ref{zerodim-3}),
\beq\label{zerodim-4}
\|\phi(f)-\sum_{i=1}^n f(\af_i)p_i\|<\ep \tforal f\in {\cal F}.
\eneq

\end{proof}

\begin{rem}\label{rem7}
Let $Y$ be a connected finite CW complex with ${\rm dim }Y=d\le 3.$
Then $[d/2]\le 1.$ Let $p\in M_n(C(Y))$ be a projection with rank $r\ge 2.$ It follows
from 8.12 of \cite{Hu} that there exists a projection $q\le p$ with rank $r-[d/2]$ which is trivial.
Thus $q\not=0.$ Then there is a rank one projection $e\le q\le p.$ Therefore $p$ is not minimal.
This shows that any minimal projection of $M_n(C(Y))$ has rank one. Since every compact metric space is an inverse
limit of finite CW complex, it is easy to see that, if $Y$ is a compact metric space with ${\rm dim Y}\le 3$ and $p\in M_n(C(Y))$
is a minimal projection, then $p(y)=0,$ or $p(y)$ has rank one for any $y\in Y.$

Let $Y$ be a compact metric space with ${\rm dim } Y>3.$ Suppose
that $C(Y, M_n)$ contains a minimal projection $p$ with rank at
least 2. Suppose that $X$ is not connected, say $X$ is a disjoint
union of two clopen subsets $X_1$ and $X_2.$ Choose a unital \hm\,
$\phi_1: C(X_1)\to (1-p)C(Y, M_n)(1-p)$ and a unital \hm\, $\phi_2:
C(X_2)\to pC(Y, M_n)p.$ Define $\phi: C(X)\to C(Y, M_n)$ by
$\phi(f)=\phi_1(f|_{X_1})+\phi_2(f|_{X_2})$ for all $f\in C(X).$
Then $\phi$ can not be possibly approximately diagonalized because
$p$ is a minimal projection. Therefore, in what follows, we mainly
consider the case that $X$ is connected, or, at least the case that
${\rm dim} X\ge 1.$

\end{rem}

\begin{prop}\label{PI1}
Let $Y$ be a compact metric space for which $\pi^1(Y)$ is trivial
and $K_1(C(Y))\not=\{0\}.$ Then there are unital \hm s from
$C(\T)\to C(Y, M_n)$ for some $n\ge 2$ which can not be
approximately diagonalized.
\end{prop}

\begin{proof}
Since $K_1(C(Y))\not=\{0\},$ there is an integer $n\ge 2$ and a
unitary $u\in C(Y, M_n)$  such that $u\not\in U_0(C(Y, M_n)).$
Define a unital \hm\, $\phi: C(\T)\to C(Y, M_n)$ by $\phi(f)=f(u)$
for all $f\in C(\T).$ Suppose that there are continuous maps $\af_k:
Y\to \T,$ $k=1,2,...,n$ and mutually orthogonal rank one projections
$\{p_1, p_2,...,p_n\}\subset C(Y, M_n)$ such that
\beq\label{PI1-1}
\|\phi(z)-\sum_{k=1}^n z(\af_k)p_k\|<1,
\eneq
where $z$ is the identity function on the unit circle $\T.$  Note
that $u=\phi(z).$
Since $\pi^1(Y)=\{0\},$ for each $k,$ there is a continuous path of
unitaries $\{w_k(t): t\in [0,1]\}\subset C(Y)$ such that
$$
w_k(0)=z(\af_k)\andeqn w_k(1)=1.
$$
One defines a continuous path of unitaries $\{U(t): t\in
[0,1]\}\subset U(C(Y, M_n))$ by
$$
U(t)=\sum_{k=1}^n w_k(t)p_k\,\,\,{\rm for}\,\,\, t\in [0,1].
$$
Then $U(0)=\sum_{k=1}^n z(\af_k)p_k$ and $U(1)=1_{M_n}.$ So
$\sum_{i=1}^n  z(\af_k)p_k\in U_0(C(Y, M_n)).$ By (\ref{PI1-1}),
$u\in U_0(C(Y, M_n)).$ A contradiction.

\end{proof}

\begin{cor}\label{C2}
There is a unital \hm\, $\phi: C(\T)\to C(S^3, M_2)$ such that
$\phi$ can not be approximately diagonalized.
\end{cor}

\begin{proof}
Let
$$
u(z, w)=\begin{pmatrix} z & -{\bar w}\\
                                    w & {\bar z},
                                    \end{pmatrix}
                                    $$
                                    where $(z, w)\in S^3=\{(z, w)\in \C^2: |z|^2+|w|^2=1\}.$
                                    Then $u\in
                                    U(M_2(C(S^3))\setminus
                                    U_0(M_2(C(S^3))).$
 However, $\pi^1(S^3)=\{0\}.$ Thus, as in the proof of \ref{PI1}, this unitary can not be approximated
 by unitaries which are diagonalized.

\end{proof}

\begin{cor}\label{C3}
Let $X$ be a  finite CW complex  with ${\rm dim} X\ge 2$ and let $Y$ be a compact
metric space for which $\pi^1(Y)$ is trivial but $K_1(C(Y))$ is not trivial. Then
there are unital \hm\, $\phi: C(X)\to C(Y, M_n)$ for some $n\ge 2$ which can not be approximately diagonalized.
\end{cor}

\begin{proof}
$X$ contains a subset $D$ which is homeomorphic to the unit disk.
Thus $X$ contains a subset $S$ which is homeomorphic to $\T.$ Define
$s: C(X)\to C(\T)$ by the restriction on $S$ and then take the
homeomorphism. If $\phi: C(\T)\to C(Y, M_n)$ is one of those unital
\hm s which can not be diagonalized (by \ref{PI1}), then $\phi\circ
s$ can not be diagonalized.

\end{proof}

\begin{lem}\label{SP}
For any $d>0,$ if $0<\dt<d,$ and if $a$ and $b$ are two normal
elements in a unital  \CA\,  such that
$$
\|a-b\|<\dt,
$$
then
$$
sp(b)\subset \{\lambda\in \C: {\rm dist}(\lambda, sp(a))<d\}.
$$

\end{lem}

The proof is standard and known. The point here is that $\dt$ does
not depend on $a$ and $b.$

\vspace{0.2in}

\begin{lem}\label{SP2}
Let $Y$ be a compact metric space and let $g\in C(Y, M_2)$ be a
normal element for which
$sp(g(y))=\{\lambda(y),\overline{\lambda(y)}\}$ for each $y\in Y.$
For any $\ep>0,$ if $\af_1, \af_2: Y\to \C$ are two continuous maps
and if $p_1, p_2\in C(Y, M_2)$ are mutually orthogonal rank one
projections such that
\beq\label{SP2-1}
\|g-(\af_1p_1+\af_2p_2)\|<\ep/8,
\eneq
then
\beq\label{SP2-2}
\|g-(\af_1p_1+{\overline{\af_1}}p_2)\|<\ep.
\eneq
\end{lem}

\begin{proof}
Suppose that (\ref{SP2-1}) holds. It follows from \ref{SP} that, for
each $y\in Y,$
\beq\label{SP2-3}
|\lambda(y)-\af_1(y)|<\ep/7\,\,\,{\rm or}\,\,\,
|\overline{\lambda(y)}-\af_1(y)|<\ep/7
\eneq
Then
\beq\label{SP2-4}
|{\overline{\lambda(y)}}-\af_2(y)\|<\ep/7\,\,\,{\rm
or}\,\,\,|\lambda(y)-\af_2(y)|<\ep/7.
\eneq
It follows that
\beq\label{SP2-5}
|\af_2(y)-\overline{\af_1(y)}|<2\ep/7
\eneq
for all $y\in Y.$ Therefore
\beq\label{SP2-6}
\|g-(\af_1p_1+\overline{\af_1}p_2)\|<\ep.
\eneq

\end{proof}

By modifying 4.4 of \cite{GP}, one has the following:

\begin{lem}\label{Count1}
Let $Y$ be a finite CW complex  with ${\rm dim} Y>3.$ Then there is
a self-adjoint element  $b\in C(Y, M_2)$  with $sp(b)=[-1,1]$ which
can not be approximately  diagonalized.
\end{lem}

\begin{proof}
Since $Y$ is a finite CW complex with ${\rm dim}Y>3,$ it contains a
4-dimensional cube. Therefore there is a subset $Y_0\subset Y$ such
that $Y_0$ is homeomorphic to $S^3.$ Identifying $Y_0$ with $S^3,$
one obtains a unitary $u\in U(M_2(C(Y_0))\setminus U_0(M_2(C(Y_0))$
which has the form
$$
u(z, w)=\begin{pmatrix} z & -{\bar w}\\
                                    w & {\bar z},
                                    \end{pmatrix}
                                    $$
                                    where $(z, w)\in S^3=\{(z, w)\in \C^2: |z|^2+|w|^2=1\}.$

In fact, for every $y\in Y_0,$ $u(y)\in SU(2).$  Since $SU(2)$ is
absolute neighborhood retract, there is a neighborhood $Y_1$ of
$Y_0$ and a map $U\in C(Y_1, SU(2))$ which extends $u.$  Let $f\in
C(Y)$ be a function such that $0\le f(y)\le 1,$ $f(y)=1$ if
$y\in Y_0$ and $f(y)=0$ if $y\in Y\setminus Y_1.$ Define a normal
element $g\in C(Y, M_2)$ by $g(y)=f(y)U(y)$ if $y\in Y_1$ and
$g(y)=0$ if $y\in Y\setminus Y_1.$ Note that the eigenvalues of $g$
have the form $\lambda$ and $\bar \lambda,$ where
$|\lambda|^2=f(x)^2.$  Let $1/3> \dt>0.$ Suppose that there are
continuous maps $\af_1,\af_2: Y\to D,$ where $D$ is the unit disk,
and mutually orthogonal rank one projections $p_1, p_2\in C(Y, M_2)$
such that
\beq\label{count1-2}
\|g-\sum_{i=1}^2 \af_i p_i\|<\dt/8.
\eneq
It follows from \ref{SP2} that
\beq\label{count1-2+}
\|g-(\af_1p_1+{\overline{\af_1}}p_2)\|<\dt.
\eneq

 Let $\pi: C(Y, M_2)\to C(Y_0, M_n)$
be the quotient map. Then
\beq\label{count1-3}
\|u(y)-(\af_1(y)\pi(p_1)+{\overline{\af_1(y)}}\pi(p_2))\|<\dt\rforal
y\in Y_0.
\eneq
Since, for each $y\in Y_0,$ $u(y)$ has eigenvalues $\lambda(y)$ and
${\overline{\lambda(y)}}$ with $|\lambda(y)|=1,$ with small $\dt,$
one may assume that
$$
|\af_1(y)-\af_1(y)|\af_1(y)|^{-1}|<1/8\tforal y\in Y_0,
$$
$i=1,2.$ It follows (with a sufficiently small $\dt$) that
\beq\label{count1-4}
\|u(y)-(\bt_1(y)\pi(p_1)+\overline{\bt_1(y)}\pi(p_2))\|<1/2
\eneq
for all $y\in Y_0,$ where $\bt_1(y)=\af_1(y)|\af_1(y)|^{-1}.$ Since
$\bt_1$ maps $Y_0$ to $S^1,$ $Y_0$ is homeomorphic to $S^3$ and
since $\pi^1(S^3)=\{0\},$ there is a continuous path of $\{w(t):
t\in [0,1]\}\subset U(C(Y_0, M_2))$ such that
\beq\label{count1-5}
w(0)=\beta_1\andeqn  w(1)=1.
\eneq
Thus $\beta_1(y)\pi(p_1)+\overline{\beta_1(y)}\pi(p_2)\subset
U_0(C(Y_0, M_2)).$ From (\ref{count1-4}) and the fact that $u\not\in
U_0(C(Y_0, M_2)),$ this is impossible.

Therefore $g$ can not be approximately diagonalized. On the other
hand, $g(y)+g^*(y)=\gamma(y)1_{M_2},$ where
$\gamma(y)=\lambda(y)+\overline{\lambda(y)},$ for all $y\in Y.$ Let
$b=(1/2 i)(g-g^*).$ Then $b$ is self-adjoint. Suppose that there
were sequences $\af_{j,n}: Y\to \R$ and sequences of pairs
$\{p_{1,n},p_{2,n}\}$ of mutually orthogonal rank one projections in
$C(Y, M_2)$ such that
\beq\label{count1-6}
\lim_{n\to\infty}\|b-(\sum_{j=1}^2 \af_{j,n}p_{j,n})\|=0.
\eneq
Note that $p_{1,n}+p_{2,n}=1_{M_2}.$  Let
$\beta_{j,n}=\gamma+\sqrt{-1}\af_{j,n},$ $j=1,2$ and $n=1,2,....$ It would
imply that
\beq\label{count1-7}
\lim_{n\to\infty}\|g-\sum_{j=1}^2 \beta_{j,n}p_{j,n}\|=0.
\eneq
This contradicts what we have proved that $g$ can not be
approximately diagonalized.  Therefore one concludes that the
self-adjoint element $b$ can not be approximately diagonalized.

Finally, since $u\in U(C(S^3, M_2))\setminus U_0(C(S^3, M_2)),$
$sp(u)=\T.$ It follows that $sp(b)=[-1,1].$

\end{proof}

\begin{thm}\label{TCount}
Let $Y$ be a finite CW complex with ${\rm dim} Y>3$ and let $n\ge 2$
be an integer. Then, for any finite CW complex  $X$ with ${\rm dim}
X\ge 1,$ there exists a unital \hm\, $\phi: C(X)\to C(Y, M_n)$ which
can not be approximately diagonalized.
\end{thm}

\begin{proof}
$X$ contains a compact subset which is homeomorphic to $[-1,1].$
There is a surjective \hm\, $s: C(X)\to C([-1,1]).$ If $n=2,$ let
$b\in C(Y, M_2)$ with $sp(b)=[-1,1]$ be the self-adjoint element
given by \ref{Count1} which can not be approximated by self-adjoint
elements in $C(Y, M_2)$ which can be diagonalized. Define $\phi:
C(X)\to C(Y, M_2)$ by $\phi(f)=s(f)(b)$ for all $f\in C(X).$ Then
$\phi$ can not be approximated by \hm s which can be diagonalized.

If $n\ge 3,$ let $\gamma: [-1,1]\to [0, 1/2]$ be a homeomorphism.
Put $\Omega_1=[0, 1/2]\cup \{1\}.$  Let $\{e_{i,j}: 1\le i,j\le n\}$
be a system of matrix units. Define $\phi_1: C(\Omega_1)\to C(Y,
M_n)$ by
\beq\label{TC-1}
\phi_1(f)=f(b)(e_{11}+e_{22})+f(1)(\sum_{i=3}^n e_{ii})\tforal f\in
C(\Omega_1).
\eneq

Suppose that $\phi_1$ that there exist continuous maps $\af_{j,k}: X\to
\Omega_{1}$ ($j=1,2,...,n$ ) and mutually orthogonal rank one
projections $p_{1,k}, p_{2, k},...,p_{n,k},$ $k=1,2,...,$ such that
\beq\label{TC-3}
\lim_{k\to \infty}\|\phi_1(f)-\sum_{j=1}^n f(\af_{j,k})p_{j,k}\|=0
\eneq
for all $f\in C(\Omega_1).$ Denote $\psi_k(f)=\sum_{j=1}^n
f(\af_{j,k})p_{j,k}$ for all $f\in C(\Omega_1).$

Let $f_0\in C(\Omega_1)$ such that $f_0(t)=0$ if $t=1$ and
$f_0(t)=1$ if $t\in [0, 1/2].$ Then $f_0$ is a projection. It
follows that
\beq\label{TC-4}
\lim_{k\to\infty}\|\phi_1(f_0)-\psi_k(f_0)\|=0.
\eneq

Note that $\phi_1(f_0)=e_{11}+e_{22}.$ From (\ref{TC-4}), when $k$
is sufficiently large, there are unitaries $v_k\in C(Y, M_n)$ such
that
\beq\label{TC-5}
v_k^*\psi_k(f_0)v_k=e_{11}+e_{22} \andeqn \lim_{k\to
\infty}\|1-v_k\|=0.
\eneq
Since $\psi_(f_0)$ is a projection which commutes with $p_{j,k}$ and
$p_{j,k}$ is a rank one projection, it follows that
$\psi_k(f_0)p_{j,k}=p_{j,k}$ or $\psi_k(f_0)p_{j,k}=0.$  By
(\ref{TC-5}), $\psi_k(f_0)$ has rank 2 for all sufficient large $k.$
To simplify notation, one may assume that
$$
\psi_k(f_0)p_{j,k}=p_{j,k},\,\,\,j=1,2.
$$
Let $q_{j,k}=v_k^*\psi_k(f_0)v_k,$ $j=1,2,$ $k=1,2,....$ Note that
$q_{j,k}$ has rank one, $j=1,2.$ By (\ref{TC-3}), (\ref{TC-4}) and
(\ref{TC-5}),
\beq\label{TC-6}
\lim_{k\to\infty} \|\phi_1(ff_0)-\sum_{j=1}^2 f(\af_{j,k})q_{j,k}\|=0
\eneq
for all $f\in C(\Omega_1).$ This would imply that $b$ could be
approximated by diagonalizable self-adjoint elements. A
contradiction.

Therefore  $\phi_1$ can not be approximated by \hm s which are
diagonalizable. There is a surjective \hm\, $s_1: C(X)\to
C(\Omega_1).$ Define $\phi: C(X)\to C(Y, M_n)$ by
$\phi(f)=\phi_1(s_1(f))$ for all $f\in C(X).$ Then $\phi$ can not be
approximated by \hm s which can be diagonalized.

\end{proof}


\vspace{0.4in}

\noindent email: hlin@uoregon.edu
\end{document}